\documentclass[12pt]{article}

\usepackage[T1]{fontenc}
\usepackage[osf]{newtxtext} 
\usepackage{array,cite,enumitem,mathtools,url,xcolor}
\usepackage[amsthm,nosymbolsc]{newtxmath} 
\usepackage[cal=euler,frak=euler]{mathalpha} 

\usepackage[margin=3cm]{geometry} 

\usepackage{titlesec}  

\usepackage{tikz}      
\usetikzlibrary{calc,intersections,patterns,shadings}

\usepackage[colorlinks]{hyperref}   
\hypersetup{linkcolor = blue,
            citecolor = green!75!black,
            urlcolor = cyan}
\usepackage{url}                    

\title{Huygens and $\pi$}

\author{Mark B. Villarino and Joseph C. Várilly
\\[9pt]
{\small
Escuela de Matemática, Universidad de Costa Rica, 
San José 11501, Costa Rica}}

\date{\today}

\DeclareMathOperator{\pgram}{parallelogram} 
\DeclareMathOperator{\polygon}{polygon}   
\DeclareMathOperator{\quadrangle}{quadrangle} 
\DeclareMathOperator{\sector}{sector}     
\DeclareMathOperator{\segment}{segment}   

\newcommand{\dl}{\delta}            
\newcommand{\Sg}{\Sigma}            
\renewcommand{\th}{\theta}          

\newcommand{\sH}{\mathcal{H}}       
\newcommand{\sM}{\mathcal{M}}       
\newcommand{\sR}{\mathcal{R}}       

\renewcommand{\leq}{\leqslant}      
\newcommand{\sst}{\scriptstyle}     
\newcommand{\tri}{\bigtriangleup\,} 
\renewcommand{\.}{\cdot}            

\newcommand{\bull}{\mathord{\scriptstyle\bullet}} 

\newcommand{\half}{{\mathchoice{\thalf}{\thalf}{\shalf}{\shalf}}} 
\newcommand{\quarter}{\tfrac{1}{4}}   
\newcommand{\shalf}{{\scriptstyle\frac{1}{2}}} 
\newcommand{\thalf}{\tfrac{1}{2}}     
\newcommand{\third}{\tfrac{1}{3}}     
\newcommand{\twothirds}{\tfrac{2}{3}} 

\newcommand{\arc}[1]{\overset{\frown}{#1}} 
\newcommand{\mot}[1]{\enspace\text{#1}\enspace} 
\newcommand{\word}[1]{\quad\text{#1}\quad} 
   
\titleformat{\section}{\normalfont\large\bfseries}
                      {\thesection}{1em}{}
\titlespacing{\section}{0pt}{*3.5}{*2.3}
\titleformat{\subsection}{\normalfont\normalsize\bfseries}
                         {\thesubsection}{0.7em}{}
\titlespacing{\subsection}{0pt}{*3.25}{*1.5}
\titleformat{\subsubsection}{\normalfont\small\bfseries}
                         {\thesubsubsection}{0.5em}{}
\titlespacing{\subsubsection}{0pt}{*3.0}{*1.2}
\titleformat{\paragraph}[runin]{\normalfont\itshape}{}{0pt}{}[.]
\titlespacing{\paragraph}{0pt}{\medskipamount}{\wordsep}

\makeatletter
\renewcommand{\@dotsep}{160} 
\makeatother

\theoremstyle{plain}
\newtheorem{thm}{Theorem}[section]  
\newtheorem{lema}[thm]{Lemma}       
\newtheorem{corl}[thm]{Corollary}   
\newtheorem{prop}[thm]{Proposition} 

\theoremstyle{remark}
\newtheorem{remk}[thm]{Remark}      

\numberwithin{equation}{section}    

\DeclareRobustCommand{\qned}{
  \ifmmode 
  \else \leavevmode\unskip\penalty9999 \mbox{}\nobreak\hfill
  \fi
  \quad\mbox{\qnedsymbol}}
\newcommand{\qnedsymbol}{$\boxminus$} 

\newcommand{\hideqed}{\renewcommand{\qed}{}} 

\begin{document}

\maketitle

\begin{abstract}
The Dutch scientist Christiaan Huygens refined Archimedes' celebrated
geometrical computation of $\pi$ to its highest point. Yet the rich
content of his beautiful treatise \textit{De circuli magnitudine
inventa} (1654) has apparently never been presented in modern form.
Here we offer a detailed and contemporary development of several of
his most striking results. We also make a historical conjecture
concerning Archimedes' trisection figure.
\end{abstract}

\tableofcontents     

\addcontentsline{toc}{section}{\vspace*{-3pc}} 


\section{Introduction} 
\label{sec:intro}

Over two thousand years ago Archimedes of Syracuse (about
275--212\,B.C.), the greatest mathematician of antiquity (and just
possibly ever) authored his monograph entitled \emph{The Measurement
of a Circle} \cite[pp.~91--98]{Heath1912}, in which he proved the
following celebrated inequality:
\begin{equation}
\boxed{3\frac{10}{71} < \pi < 3\frac{1}{7}}\,.
\label{eq:Arch1} 
\end{equation}
(Both bounds are accurate to \emph{two} decimal places.)

We translate Archimedes' own statement of \eqref{eq:Arch1} since it
might surprise the modern reader:

\begin{thm}
The circumference of any circle is greater than three times the
diameter and exceeds it by a quantity less than a seventh part of the
diameter but greater than ten seventy-first parts.
\end{thm}

Observe that there is \emph{no mention of the constant}~$\pi$. Indeed,
although Euclid proved in Proposition XII.2 of \textit{Elements} that
the ratio [of areas] of a circle to the square on its diameter is
constant, Greek mathematicians never had a special name or notation
for~it.

We recall that Archimedes proved \eqref{eq:Arch1} in three steps:
\begin{itemize}
\item
First, he ``compressed'' a circle between inscribed and
circumscribed regular \textit{hexagons}.

\item
Second, he proved (in geometric guise) the identity
\begin{equation}
\cot\Bigl( \frac{\th}{2} \Bigr) = \cot\th + \sqrt{\cot^2\th + 1}.
\label{eq:trig} 
\end{equation}

\item
Third, he used \eqref{eq:trig} to \textit{recursively} compute the
perimeters of inscribed and circumscribed regular $12$-gons,
$24$-gons, $48$-gons, and finally $96$-gons.
\end{itemize}

For the circumscribed $96$-gon, Archimedes obtained \cite{Heath1912}:              
\begin{equation}
\frac{\text{perimeter of $96$-gon}}{\text{diameter}}
< \frac{14688}{4673\half} = 3 + \frac{667\half}{4673\half}
< 3 + \frac{667\half}{4672\half} = 3\frac{1}{7}\,.
\label{eq:upper} 
\end{equation}
For the inscribed $96$-gon, he obtained
\begin{equation}
\frac{\text{perimeter of $96$-gon}}{\text{diameter}}
> \frac{6336}{2017\frac{1}{4}}
= 3 + \frac{10}{71} + \frac{37}{572899} > 3\frac{10}{71}\,.
\label{eq:lower} 
\end{equation}
 
Subsequent generations of mathematicians sought closer bounds on~$\pi$
by increasing the number of sides of the inscribed and circumscribed
polygons but, until the advent of symbolic algebra and calculus,
Archimedes' procedure remained the paradigm \textit{for almost two
thousand years!}

Some 1900 years later, in 1654, the 25-year-old Dutch mathematician
Christiaan Huygens applied a bewildering \textit{tour de force} of
elementary geometry to the perimeter of an inscribed regular $60$-gon
to obtain the following spectacular improvement of~\eqref{eq:Arch1}:
\begin{equation}
\boxed{\strut 3.1415926533 < \pi < 3.1415936538}\,.
\label{eq:Huygens-estimate} 
\end{equation}

He did so in the last section of his brilliant treatise \textit{De
circuli magnitudine inventa} \cite{Huygens1654} which, for brevity, we
shall simply call \textit{Inventa}. In the preface to
\textit{Inventa}, Huygens declares that (up to his time) the only
theorem in circle quadrature with a proper proof states that the
perimeter of a circle is bounded above and below by the perimeters of
a circumscribed and inscribed polygon, respectively, and that his
treatise will do more.

His assessment is rather modest since \textit{Inventa} not only is
epoch-making for circle quadrature -- as shown, for example, by
\eqref{eq:Huygens-estimate} -- but is also one of the most beautiful
and important elementary geometric works ever written, and like the
\textit{Measurement} of Archimedes, will retain its value even if its
results can be now obtained much more quickly using modern analysis.

\textit{Inventa} not only contains the estimates
\eqref{eq:Huygens-estimate}, but also the first proofs in the history
of mathematics of the famous inequalities:
\begin{equation}
\boxed{\frac{3 \sin x}{2 + \cos x} \leq x 
\leq \frac{2}{3} \sin x + \frac{1}{3} \tan x}
\label{eq:Cusa-and-Snell} 
\end{equation}
where $0 \leq x \leq \frac{\pi}{2}$.

Huygens proves them (geometrically!) as Theorems XII and~XIII of his
treatise. The first proof of any famous theorem acquires a historical
and methodological importance, all the more so if the theorem had
remained stated but unproved for some time. The lower bound is
apparently due to Nikolaus of Cusa \cite{Cusanus1514} while the upper
bound is due to Snell (Willebrord Snellius) \cite{Snell1621}, but
neither gave a rigorous proof. Huygens asserted, correctly, that his
treatise presented the proofs which were lacking.

Notwithstanding their antiquity, the inequalities of Cusa and Snell
continue to be a source of contemporary research. See, for example,
\cite{BhayoS2015, Nelsen2020, Chouikha2022}.

Moreover, Huygens shows his wonderful originality by transforming
Archimedes' exact determination of the barycenter of a parabolic
segment into \textit{barycentric inequalities} 
for circular segments.
The latter become even more precise \textit{arc-length inequalities},
namely:
\begin{gather}
\boxed{x < \sin x + \frac{10(4\sin^2\frac{x}{2} - \sin^2 x)}
{12\sin\frac{x}{2} + 9\sin x}}
\label{eq:Huygens-upper} 
\\
\shortintertext{and}
\boxed{x > \sin x + \frac{10(4\sin^2\frac{x}{2} - \sin^2 x)}
{12\sin\frac{x}{2} + 9\sin x + 8\dfrac{(2\sin\frac{x}{2} - \sin x)^2}
{12\sin\frac{x}{2} + 9\sin x}}}
\label{eq:Huygens-lower} 
\end{gather}
both valid in the range $0 < x < \pi$.

\goodbreak 

These two barycentric inequalities are the tools that Huygens uses to
prove the results \eqref{eq:Huygens-estimate} at the end of his
\emph{Inventa}.%
\footnote{%
Huygens' \emph{Inventa} comprises 20 ``propositions'', consisting of
16 theorems interleaved with 4 ``problems''. With this dual labeling,
the last section is called Problem~IV and also Proposition~XX.}
This is by no means straightforward, but he carries out in great
detail the proof of the first inequality \eqref{eq:Huygens-upper}.
 
He never published a proof of the second inequality
\eqref{eq:Huygens-lower}. It can be proved without difficulty with
modern techniques; this was first achieved, to our knowledge, by Iosif
Pinelis in a discussion with one of us on \textit{Math Overflow}
\cite{Pinelis2024}.
 
We also point out that numerical analysts describe the three formulas
\eqref{eq:Cusa-and-Snell}, \eqref{eq:Huygens-upper},
\eqref{eq:Huygens-lower} as historically the first examples of
\textit{Richardson extrapolation}, in which suitable algebraic
combinations of simple functions produce highly accurate
approximations. These examples appeared some three centuries before
the papers of L. Richardson who originated the modern method
\cite{Richardson1927}.

It is not easy for the modern reader to find a good presentation of
Huygens' beautiful geometrical proofs; one must consult the
\textit{Inventa} itself. Sadly, even the \textit{Inventa} presents
difficulties to today's reader, for Huygens chose to present
\textit{Inventa} in a strict Euclidean--Archimedean format:
\begin{itemize}[nosep]
\item
The proofs are entirely synthetic, without contextual analysis or
motivation.
\item
There is no algebraic notation; proportions and inequalities are
written out as long sentences in words.
\item
The main steps are not set out separately. Hence, intricate proofs
comprise a single paragraph of several pages of running printed text.
\end{itemize}
These criticisms of format also apply to the French
\cite{Huygens-v12}, German \cite{Rudio1892}, and -- somewhat less~so
-- to the English translation \cite{Pearson1923} of \textit{Inventa},
since they are strictly literal.

A few studies of Huygens' work are available. For instance, the long
paper by Schuh \cite{Schuh1914}, from~1914, covers some of the
thematic ground. Later, an important paper by Hoffman
\cite{Hofmann1966} considered the historical sources of Huygens'
\emph{Inventa} and gave a general description of its contents. Both
authors correctly asserted that the idea behind Huygens' treatise was
to approximate a circular segment by a parabolic one. Hofmann briefly
described Huygens' proofs of the Cusa and Snell inequalities and
discussed the ideas behind Huygens' barycentric inequalities. However,
his purpose differs from ours in that he replaced many of Huygens'
proofs with his own, including a calculus proof of what we call the
Grossehilfsatz (see Appendix~\ref{app:helpful}). In this paper, we
resurrect Huygens' own original proofs, in some detail, so that the
modern reader can appreciate their geometric beauty.

Our exposition presents Huygens' proofs of \eqref{eq:Cusa-and-Snell}
and makes some of the rich content of \textit{Inventa} available in
a modern form.

\paragraph{The fundamental idea in Huygens' tract}
is simple and brilliant: Huygens approximates the area of a
\textit{circular} segment (and thus, sector) by the area of a suitable
\textit{parabolic} segment.

Why?

Because Archimedes exhaustively investigated the metrical properties
of a parabolic segment, and so Huygens is able to transform
Archimedes' \textit{exact equations} for the area of a parabolic
segment into \textit{inequalities} for the area of a circular segment.
But the area of a circular segment is a simple function of its radius
and its arc length. Therefore, the area inequalities become the
aforementioned \textit{arc length inequalities}.

The same idea undergirds Huygens' barycentric inequalities: the exact
position of the barycenter of a parabolic segment, as determined by
Archimedes, becomes the approximate location of the barycenter of a
\textit{circular} segment. Since Huygens determined the exact location
of the latter, the relation between the two barycenters yields an
arc-length inequality.

\paragraph{Plan of the paper}
Section~\ref{sec:Cusa-and-Snell} exposits Huygens' proofs of Nikolaus
of Cusa's \emph{lower} bound on the circumference of a circle, and
Snell's \emph{upper} bound. Both bounds were originally proposed
without proof by the mentioned authors.

Section~\ref{sec:barycenter} presents Huygens' proofs of his own
original barycentric inequalities, the last of which Huygens stated
without proof. These inequalities allowed Huygens to obtain upper and
lower bounds on~$\pi$ with \textit{nine} decimal figures of accuracy
using only an inscribed $60$-gon.

In the short Section~\ref{sec:speculation}, we offer the conjecture,
based on the famous Archimedes' \emph{trisection figure}, that
Archimedes may have used (an equivalent of) the Snell--Cusa
convergence-improving inequalities to compute more accurate bounds
on~$\pi$.

There are two appendices. In Appendix~\ref{app:helpful}, we retrieve
the very Archimedean proof of a theorem of Huygens from an earlier
work~\cite{Huygens1651} that provides the essential step in locating
the barycenter of a circular segment. Indeed, in
subsection~\ref{ssc:barycenter-location} we prove that this discovery
of Huygens predates the modern formula for its location. In 
Appendix~\ref{app:area-comparison}, we determine an area inequality
inspired by the work of Schuh and Hofmann.

\begin{figure}[htb]
\centering
\begin{tikzpicture}[scale=0.8, rotate=-15]
\coordinate (O) at (0,0) ;
\coordinate (A) at (0:2cm) ;
\coordinate (B) at (70:2cm) ;
\coordinate (C) at (140:2cm) ;
\coordinate (D) at ($ (A)!(B)!(C) $) ;
\coordinate (P) at (-110:2cm) ;
\coordinate (R) at (-40:2cm) ;
\coordinate (M) at ($ (C)!(A)!(R) $) ;
\draw (O) circle(2cm) node[left] {$O$} ;
\draw[fill=gray!20] (A) arc(0:140:2cm) -- cycle ; 
\draw (A) node[below right] {$A$}
    -- (B) node[above right] {$B$}
    -- (C) node[above left] {$C$} ;
\draw (B) -- (D) node[right=1pt] {$D$} ;
\draw (A) -- (M) node[left] {$M$} ;
\draw (D) -- (P) node[below left] {$P$} ;
\draw (C) -- (R) node[below right] {$R$} ;
\draw[gray] (M) ++(50:2mm) -- ++(140:2mm) -- ++(230:2mm) ; 
\foreach \pt in {A,B,C,D,M,O,P,R} \draw (\pt) node{$\bull$} ;
\end{tikzpicture}
\caption{A circular segment less than a semicircle}
\label{fg:notation} 
\end{figure}

\paragraph{Some terminology}
In what follows, the term \textit{segment} will always refer to
the plane figure enclosed by an arc of a curve (either a circle or a
parabola) and the chord joining the endpoints of the arc. (Line
segments are never in question.)

A \textit{maximum triangle} in a circular segment is one whose base
and altitude are the same as those of the segment. Such a triangle is
isosceles. Its apex is the midpoint of the circular arc, called the
\textit{vertex} of the circular segment, i.e., the point on its arc
most distant from its chord.

The \textit{diameter} of a circular segment is the perpendicular
bisector of the base chord; as such, it is a part of a diameter of the
full circle. (The term is not used to denote the maximal distance
between two points of the segment.) The length of the diameter is 
called the \textit{height} of the segment.

The \textit{area} of a plane figure is denoted by parentheses; thus 
$(\tri ABC)$ and $(\segment ABC)$ are the areas of the triangle
$ABC$ and the segment $ABC$, respectively.

The \textit{sine} of a circular arc (less than a semicircle) is the
length of the perpendicular from one extremity of the arc to the
circle's diameter through the other extremity. (This is Huygens'
terminology; in modern language, that would be the sine of the central
angle times the radius of the circle.)

In Figure~\ref{fg:notation}, the shaded area is the circular segment
$ABC$, with vertex $B$, base or chord $AC$, and diameter $BD$. The two
diameters of the circle $BP$ and $CR$ meet at its center~$O$. The
perpendicular $AM$ from $A$ to the diameter $CR$ is the sine of the
arc~$\arc{AC}$.


\section{Two early bounds on the circumference} 
\label{sec:Cusa-and-Snell}

Huygens bases the first part of his tract \textit{Inventa} (up to and
including its Theorem~XIII) on two propositions giving upper and lower
bounds for the area of a \textit{circular segment }(and therefore,
sector), clearly inspired by Archimedes' corresponding proposition on
\textit{parabolic} segments.

\subsection{The first Heron--Huygens lemma} 
\label{ssc:Heron-Huygens-one}

In his essay \textit{The quadrature of the parabola}
\cite[p.~246]{Heath1912}, Archimedes famously proved the following
area relation.

\begin{thm} 
\label{th:parabola-quadrature}
Every segment bounded by a parabola and a chord is equal to
\textbf{four-thirds} of the triangle which has the same base and
height as the segment.
\end{thm}

Archimedes' proof exhibits the first explicit sum of an infinite
(geometric) series in the history of mathematics (of ratio
$\frac{1}{4}$).

It is of historical interest that Heron of Alexandria
\textit{anticipated Huygens by some 1600 years} in section~32 of his
treatise \textit{Metrika} \cite{Heron}. Thought to be irretrievably
lost for almost two thousand years, the manuscript \textit{Metrika}
was rediscovered in 1896, some 250 years after the writing of
\textit{Inventa}. Therefore, Huygens was quite unaware of Heron's work
when he did his own investigation. Nevertheless, the statement of
Huygens' Theorem~I (immediately below) coincides almost word for word
with Heron's. For this reason we call the two bounds the
\textit{Heron--Huygens lemmas}.

\begin{lema} 
\label{lm:Heron-Huygens-1}
If in a segment of a circle less than a semicircle a maximum triangle
be inscribed, and in the subtended segments triangles be similarly
inscribed, the triangle first drawn will be less (in area) than
\textbf{four times} the sum of the two which were drawn in the
subtended segments.
\end{lema}

\begin{figure}[htb]
\centering
\begin{tikzpicture}[scale=0.8, rotate=-30]
\coordinate (O) at (0,0) ;
\coordinate (A) at (0:2cm) ;
\coordinate (B) at (80:2cm) ;
\coordinate (C) at (160:2cm) ;
\coordinate (D) at ($ (A)!0.5!(C) $) ;
\coordinate (E) at (40:2cm) ;
\coordinate (F) at (120:2cm) ;
\coordinate (G) at (intersection of B--D and E--F) ;
\coordinate (P) at (-100:2cm) ;
\draw[dashed] (O) circle(2cm) ;
\draw (A) arc(0:160:2cm) -- cycle ;  
\draw (A) node[below right] {$A$}
    -- (B) node[above right] {$B$}
    -- (C) node[above left] {$C$} ;
\draw (A) -- (E) -- (B) -- (F) -- (C) ;
\draw (E) node[right] {$E$} -- (F) node[above] {$F$} ;
\draw[dashed] (B) -- (G) node[below=2pt] {$G$}
    -- (D) node[below=2pt] {$D$} -- (P) node[below left] {$P$}
    -- (A) ;
\foreach \pt in {A,B,C,E,F} \draw (\pt) node{$\bull$} ;
\foreach \pt in {D,G,P} \draw (\pt) node{$\sst\circ$} ;
\end{tikzpicture}
\caption{Maximum triangles in circle segments}
\label{fg:Heron-Huygens-1} 
\end{figure}

\begin{proof}
Given the segment $ABC$ of a circle, and $BD$ the diameter of the
segment, let there be inscribed a maximum triangle $\tri ABC$
(see Figure~\ref{fg:Heron-Huygens-1}). Likewise, in the two subtended
segments let there be inscribed maximum triangles $\tri AEB$ and
$\tri BFC$. It is required to prove that the \textit{area}
$(\tri ABC)$ is less than \textit{four times} the sum of
$(\tri AEB)$ and~$(\tri BFC)$.

Let $EF$ be joined, cutting the diameter of the segment at the
point~$G$. Then there are \textit{three congruent triangles}, namely
$\tri AEB$, $\tri BFC$, and the new triangle
$\tri EBF$: see Figure~\ref{fg:Heron-Huygens-1}. We shall prove
that
\begin{equation}
(\tri ABC) < 8(\tri EBF),
\label{eq:eighth} 
\end{equation}
which will establish the lemma, since
$8(\tri EBF) = 4(\tri AEB) + 4(\tri BFC)$.

\medskip

We claim that the heights of $\tri ABC$ and $\tri EBF$ are
related by:
\begin{equation}
BD < 4\,BG.
\label{eq:both-heights} 
\end{equation}
To see that, notice that the arc $\arc{AB}$ is bisected by the
point~$E$. Therefore
$$
EA \ (\text{or } EB) > \frac{1}{2} AB
\word{and thus}  AB^2 < 4\, EB^2 \text{ or } 4\, EA^2,
$$
which implies
$$
\frac{BD}{BG} = \frac{BD \. BP}{BG \. BP} = \frac{AB^2}{EB^2} < 4,
$$
if $BP$ is the diameter of the circle through the point $B$, since
$AB^2 = BD \. BP$ and $EB^2 = BG \. BP$. The first of these
equalities comes from the similar triangles
$\tri DBA \sim \tri ABP$, which entails $DB : AB = AB : BP$. The other
equality follows likewise from $\tri EBG \sim \tri PBE$.

Next, the bases of $\tri ABC$ and $\tri EBF$ obey:
\begin{equation}
AC < 2\,EF.
\label{eq:both-bases} 
\end{equation}
This follows from the triangle inequality:
$$
EF = AB = BC  \implies  2\,EF = AB + BC > AC.
$$
Together, \eqref{eq:both-heights} and \eqref{eq:both-bases} yield the
desired inequality \eqref{eq:eighth}.
\end{proof}

\begin{lema}[Heron--Huygens I: Theorem III of \emph{Inventa}] 
\label{lm:Huygens-III}
The area of a circular segment less than a semicircle has a
\textbf{greater ratio} to the area of its maximum inscribed triangle
\textbf{than four to three}.
\end{lema}

\begin{proof}
In the circular segments of Figure~\ref{fg:Heron-Huygens-1} whose
bases are the chords $AE, EB, BF, FC$, we again draw the maximum
isosceles triangles, and then apply the process again, etc., and so
fill out the segment. Applying Lemma~\ref{lm:Heron-Huygens-1}, we
obtain
\begin{equation}
\boxed{(\segment ABC) > \frac{4}{3}\,(\tri ABC)} 
\label{eq:HH-first} 
\end{equation}
from:
\begin{align*}
(\segment ABC) 
&= (\tri ABC) + 2(\tri AEB) + 4(\cdots) + \cdots
\\
&> (\tri ABC) 
\Bigl( 1 + \frac{1}{4} + \frac{1}{4^2} +\cdots \Bigr)
= \frac{4}{3}\,(\tri ABC).
\tag*{\qed} 
\end{align*}
\hideqed
\end{proof}

One can see that this is almost exactly the proof that Archimedes
gives for the area of a \textit{parabolic} segment. The only
difference is that equality is replaced by \textit{inequality} since
the parabolic segment is \textit{smaller} than the corresponding
circular segment. Moreover, Heron proves this theorem in almost
exactly the same way.


\subsection{Huygens' first arc length inequality} 
\label{ssc:Huygens-first-inequality}

Huygens' first inequality is a \textit{lower} bound for the
circumference of a circle. Here we see how Huygens proves inequalities
about circular and polygonal \textit{areas}, and then translates them
into inequalities about \textit{perimeters}.

\begin{thm}[First arc length inequality:
Theorem VII of \emph{Inventa}] 
\label{th:Huygens-VII}
If $C$ denotes the circumference of a circle, then
\begin{equation}
\boxed{C > C_{2n} + \frac{1}{3}(C_{2n} - C_n)}
\label{eq:Snell-first} 
\end{equation}
where $C_n$ denotes the perimeter of the inscribed regular polygon of
$n$~sides.
\end{thm}

\begin{figure}[htb]
\centering
\begin{tikzpicture}[scale=0.8, rotate=60]
\coordinate (O) at (0,0) ;
\coordinate (A) at (60:3cm) ;
\coordinate (B) at (30:3cm) ;
\coordinate (C) at (0:3cm) ;
\coordinate (E) at (-60:3cm) ;
\coordinate (M) at ($ (O)!(A)!(C) $) ;
\coordinate (Ac) at ($ (A)!0.6!(C) $) ;
\coordinate (Am) at ($ (A)!0.55!(M) $) ;
\coordinate (Ba) at (36:3cm) ;
\coordinate (Oa) at ($ (O)!0.5!(A) $) ;
\draw (C) arc(0:60:3cm) ;           
\draw[dashed] (C) arc(0:-60:3cm) ;  
\draw (C) node[above right] {$C$}
    -- (A) node[above left] {$A$}
    -- (M) node[below right=-2pt] {$M$} ;
\draw[dashed] (M) -- (E) node[below right=-1pt] {$E$} ; 
\draw (C) -- (O) node[below left=-1pt] {$O$} -- (A) -- cycle ;
\draw (C) -- (B) node[above] {$B$} -- (A) ;
\draw (O) -- (B) ;
\draw[gray] (M) ++(0.2,0) -- ++(0,0.2) -- ++ (-0.2,0) ; 
\draw (Ac) node[below=-1pt] {$\sst b$} ;
\draw (Am) node[below left=-2pt] {$\sst c$} ;
\draw (Ba) node[above left] {$\sst s$} ;
\draw (Oa) node[below left=-1pt] {$\sst r$} ;
\foreach \pt in {A,C,B,E,M,O} \draw (\pt) node{$\bull$} ;
\end{tikzpicture}
\caption{Polygon perimeter comparisons}
\label{fg:Huygens-VII} 
\end{figure}

\begin{proof}
We first obtain an inequality for a \emph{sector} and its arc. In a
circle with center~$O$ and radius~$r$, let $ACE$ be a segment, with
diameter $CM$ and base $AE$. We set $b := AC$, $c := AM$ and let $s$
be the length of the circular arc $\arc{AC}$, whose midpoint is $B$.
See Figure~\ref{fg:Huygens-VII}. We shall show that
\begin{equation}
\arc{AC} = s > b + \frac{1}{3}(b - c)
= AC + \frac{1}{3}(AC - AM).
\label{eq:Snell-any-arc} 
\end{equation}

\goodbreak 

We now add the area $(\tri OAC) = \half cr$ to that of the
circular segment $ABC$ to form the circular \textit{sector} of area
$(\sector OAC) = \half rs$. Thus
$$
(\tri ABC) = 2(\tri OBC) - (\tri OAC) 
= (b - c) \. \frac{r}{2}\,.
$$
The first Heron--Huygens lemma (Lemma~\ref{lm:Huygens-III}), applied 
to the segment $ABC$, now implies 
$$
\frac{4}{3}(b - c) \. \frac{r}{2} + \frac{cr}{2}
< (\segment ABC) + (\tri OAC) = (\sector OAC) = \frac{rs}{2} \,,
$$
and therefore
\begin{equation}
b + \frac{1}{3}(b - c) = \frac{4b - c}{3} = c + \frac{4}{3}(b - c) < s
\label{eq:easy-lower-bound} 
\end{equation}
which establishes \eqref{eq:Snell-any-arc}.

Now consider the special case where $b = AC$ is a side of an inscribed
regular polygon of $2n$~sides; then $2c = 2\,AM = AE$ will be a side
of the inscribed regular polygon of $n$~sides. On multiplying
\eqref{eq:Snell-any-arc} by~$2n$, we arrive at Snell's
inequality~\eqref{eq:Snell-first}.
\end{proof}

\paragraph{Error Analysis}
We observe, as in a footnote in Huygens' \textit{Oeuvres}
\cite[p.~128]{Huygens-v12}, that if $p_n$ denotes the perimeter of an
inscribed regular $n$-gon in a circle of radius~$1$, the central angle
of the $n$-gon is $\frac{2\pi}{n}$ and we obtain the following Taylor
expansion for \textit{Huygens' first inequality}:
$$
p_n = 2n \sin\Bigl( \frac{\pi}{n} \Bigr)
\implies p_{2n} + \frac{1}{3}(p_{2n} - p_n)
= 2\pi - \frac{\pi^5}{240n^4} + \frac{\pi^7}{80640n^6} -\cdots
$$
which shows not only that the approximation is in \textit{defect}, but
also that the error does not exceed
$\dfrac{\pi^5}{240n^4} \doteq \dfrac{1.275\cdots}{n^4}\,$, which
clearly is quite small for large~$n$.


\subsection{Nikolaus of Cusa's lower bound} 
\label{ssc:Cusa-lower-bound}

\begin{thm}[Nikolaus of Cusa: Theorem XIII of \emph{Inventa}] 
\label{th:Huygens-XIII}
If to the diameter of a circle a radius is added in the same
direction, and a line is drawn from the end of the extended line
cutting the circle and meeting the tangent to the circle at the
opposite extremity of the diameter, this will intercept a part of the
tangent \textbf{less} than the adjacent intercepted \textbf{arc}.
\end{thm}

\begin{figure}[htb]
\centering
\begin{tikzpicture}
\coordinate (O) at (0,0) ;
\coordinate (A) at (-2,0) ;
\coordinate (B) at (2,0) ;
\coordinate (C) at (-4,0) ;
\coordinate (Ch) at (-5.5,0) ;
\coordinate (H) at (-5,0) ;
\coordinate (Ld) at (2,-1.8) ; \coordinate (Lu) at (2,3.5) ;
\draw[name path=bigOh] (O) circle(2cm) ;
\path[name path=bigArc] (A) circle(3cm) ;
\path[name path=Ab] (A) -- (B) ; 
\path[name intersections={of=bigArc and bigOh, by={E,E2}}] ;
\coordinate (L) at (intersection of C--E and B--Lu) ;
\coordinate (D) at ($ (B)!(E)!(L) $) ;
\coordinate (G) at ($ (A)!(E)!(B) $) ;
\coordinate (K) at (intersection of E--H and Ld--Lu) ;
\draw (C) node[below] {$C$} -- (A)  node[below left] {$A$}
    -- (B) node[right] {$B$} -- (L) node[above right] {$L$} -- cycle ;
\draw (Ld) -- (Lu)   (Ch) -- (C) ;
\draw[dashed] (D) node[right] {$D$} -- (E) node[above left] {$E$}
    -- (G) node[below] {$G$} ;
\draw (H) node[below right] {$H$} -- (E) -- (K) node[right] {$K$} ;
\draw (O) node[below left] {$O$} -- (E) ;
\draw[dashed] (A) -- (E) -- (B) ;
\draw[gray] (B) ++(-0.2,0) -- ++(0,-0.2) -- ++ (0.2,0) ; 
\draw[gray] (D) ++(0,-0.2) -- ++(-0.2,0) -- ++ (0,0.2) ; %
\draw[gray] (G) ++(0.2,0) -- ++(0,0.2) -- ++ (-0.2,0) ;  
\foreach \pt in {A,B,C,D,E,G,H,K,L,O} \draw (\pt) node{$\bull$} ;
\end{tikzpicture}
\caption{Tangent and arc length comparison}
\label{fg:Huygens-XIII} 
\end{figure}

\begin{proof}
Given a circle with diameter $AB$, let $AB$ be produced to the point
$C$ such that $AC$ is equal to the radius. Let $CL$ be drawn cutting
the circumference the second time in $E$ and meeting at~$L$ the
tangent to the circle at the extremity $B$ of the diameter. (See
Figure~\ref{fg:Huygens-XIII}.) One must prove that
\begin{equation}
\boxed{BL < \arc{EB}}\,.
\label{eq:Snell-second} 
\end{equation}

Draw $AE$ and $EB$, and locate $H$ on the prolongation of the diameter
$AB$ beyond $A$ so that $AH = AE$. Produce $HE$ to meet the tangent
line $BL$ at~$K$. Finally, draw perpendiculars $EG$ from $E$ to the
diameter $AB$ and $ED$ from~$E$ to the tangent $BL$.

Now $\angle EHA = \angle HEA$ since $\tri HAE$ is isosceles; and since
$\angle AEB = \frac{\pi}{2}$, it follows that
$$
\angle HEA + \angle KEB = \frac{\pi}{2} \,.
$$
Now, $\angle HBK = \frac{\pi}{2}$; hence, in the triangle $\tri HKB$,
$$
\angle KHB + \angle BKH = \frac{\pi}{2} \,.
$$
Subtracting equals from equals, $\angle HEA$ on one side and
$\angle KHB = \angle EHA$ on the other, we conclude that
$\angle KEB = \angle BKH = \angle BKE$. Therefore, the triangle
$\tri KEB$ is isosceles, with $BE = BK$.

Moreover, $BD = EG$ as sides of the rectangle $BDEG$, which implies
$$
DK = BE - EG.
$$
Next,
$$
\frac{AG}{AE} = \frac{AE}{AB}  \implies  \frac{AB + AG}{2} > AE,
$$
i.e., the arithmetic mean of $AB$ and $AG$ is greater than their
geometric mean; it follows that
$$
AH = AE < \frac{1}{2}(AB + AG) = CA + \frac{1}{2} AG,
$$
and subtracting $CA$ from both sides%
\footnote{%
This subtraction is valid because $AH > CH$; or equivalently, $AE$ is 
longer than the radius of the circle. That happens because the point 
$E$ is chosen as the \textit{second} intersection of the line from 
$C$ to~$L$ with the circumference of the circle. Another consequence
is that $C$ lies between $B$ and~$H$, so that $K$ lies between $B$ 
and~$L$, as shown in Figure~\ref{fg:Huygens-XIII}.}
yields $CH < \half AG$.

But $CA = \half AB > \half AG$, and on adding $CA$ to $AG$, one 
deduces that
\begin{equation}
CG > \frac{3}{2} AG > 3 CH.
\label{eq:crucial-step} 
\end{equation}

But, since
$$
\frac{GH}{EG} = \frac{DE}{DK}   \word{and}
\frac{EG}{CG} = \frac{DL}{DE}
$$
by similar triangles $\tri EHG \sim \tri KED$ and
$\tri ECG \sim \tri LED$, then, by multiplying these two ratios, one
deduces
$$
\frac{GH}{CG} = \frac{DL}{DK} \word{and consequently}
\frac{CH}{CG} = \frac{KL}{DK}\,, \word{whereby} 
\frac{CG}{CH} = \frac{DK}{KL} \,.
$$
Thus, on account of~\eqref{eq:crucial-step},
$$
BE - EG = BK - BD = DK > 3 KL, \word{and thus}
KL < \frac{1}{3}(BE - EG).
$$
Then, invoking the inequality \eqref{eq:Snell-any-arc} from the proof
of Theorem~\ref{th:Huygens-VII}, i.e., Theorem~VII of \emph{Inventa},
we obtain the desired inequality~\eqref{eq:Snell-second}:
$$
BL = BK + KL = BE + KL < BE + \frac{1}{3}(BE - EG) < \arc{BE}.
\eqno \qed 
$$
\hideqed
\end{proof}

\paragraph{Error Analysis}
If we take the radius $OA = 1$, and put $x := \arc{BE}$, then by
similar triangles $\triangle{LBC} \sim \triangle{EGC}$, we get
$$
\frac{BL}{3} = \frac{BL}{BC} = \frac{EG}{CG} = \frac{EG}{2 + OG}\,;
\word{in other words,} \frac{BL}{3} = \frac{\sin x}{2 + \cos x}\,.
$$
Now, we just proved that $x > BL$; hence,
\begin{equation}
\boxed{x > \frac{3\sin x}{2 + \cos x}}\,.
\label{eq:modern-Cusa-inequality} 
\end{equation}
The formula \eqref{eq:modern-Cusa-inequality} \textit{is the modern
statement of Cusa's inequality}. Its accuracy is given by
$$
\frac{3\sin x}{2 + \cos x} 
= x - \frac{x^5}{180} - \frac{x^7}{1512} -\cdots
$$
which shows Cusa's approximation
$$
x \doteq \frac{3\sin x}{2 + \cos x}
$$
to be in \textit{defect}; and if we put $x := \frac{\pi}{2n}$, we
conclude that
$$
2\pi = 4n \. BL + \frac{\pi^5}{1440 n^4} +\cdots
$$
with an error $\dfrac{\pi^5}{1440n^4} = \dfrac{0.2125\cdots}{n^4}$
which is quite small for large~$n$.


\subsection{The second Heron--Huygens lemma} 
\label{ssc:Heron-Huygens-two}

We now seek to prove an \textit{upper} bound for the circumference.

\begin{lema} 
\label{lm:Heron-Huygens-2}
If a triangle is drawn having the same base as a segment of circle
less than a semicircle and having its sides \textbf{tangent} to the
segment, and if a line is drawn tangent to the segment at its vertex,
this cuts off from the given triangle a triangle greater than
\textbf{one half} of the maximum triangle described within the
segment.
\end{lema}

\begin{figure}[htb]
\centering
\begin{tikzpicture}
\coordinate (A) at (0:2cm) ;
\coordinate (B) at (60:2cm) ;
\coordinate (C) at (120:2cm) ;
\draw (A) ++ (0,1) coordinate (Ae) ;
\draw (B) ++ (-30:1cm) coordinate (Bf) ;
\draw (C) ++ (30:1cm) coordinate (Ce) ;
\coordinate (D) at ($ (A)!0.5!(C) $) ;
\coordinate (E) at (intersection of A--Ae and C--Ce) ;
\coordinate (F) at (intersection of A--E and B--Bf) ;
\coordinate (G) at (intersection of C--E and B--Bf) ;
\coordinate (J) at (30:2cm) ;
\draw (J) ++ (-60:1cm) coordinate (Jh) ;
\coordinate (H) at (intersection of A--E and J--Jh) ;
\coordinate (K) at (intersection of F--G and J--Jh) ;
\draw (A) arc(0:120:2cm) -- cycle ;  
\draw (A) node[right] {$A$}
    -- (B) node[above right] {$B$}
    -- (C) node[above left] {$C$} ;
\draw (A) -- (E) node[right] {$E$} -- (C) ;
\draw (F) node[right] {$F$} -- (G) node[above left] {$G$} ;
\draw (H) node[right] {$H$} -- (K) node[above right=-2pt] {$K$} ;
\draw[dashed] (E) -- (D) node[below left] {$D$} ;
\draw[dashed] (A) -- (J) node[below left=-2pt] {$J$} -- (B) ;
\draw[gray] (D) ++(-30:2mm) -- ++(60:2mm) -- ++ (150:2mm) ;
\foreach \pt in {A,B,C,D,E,F,G} \draw (\pt) node{$\bull$} ;
\foreach \mk in {H,J,K} \draw (\mk) node{$\sst\circ$} ;
\end{tikzpicture}
\caption{Outer and inner triangles of a circle segment}
\label{fg:Heron-Huygens-2} 
\end{figure}

\begin{proof}
Take a circular segment $ABC$ less than a semicircle with its vertex
at~$B$, and let the lines $AE$ and $CE$, tangents to the segment at
the extremities of its base, meet at~$E$. (They will indeed meet,
since the segment is less than a semicircle.) Moreover, let the line
$FG$, with $F$ on~$EA$ and $G$ on~$EC$, be drawn tangent to the
segment at its vertex~$B$; join $AB$ and $BC$. (See
Figure~\ref{fg:Heron-Huygens-2}.) We must prove:
\begin{equation*}
\boxed{(\tri FEG) > \frac{1}{2} (\tri ABC)}\,.
\end{equation*}
Clearly $\tri AEC$, $\tri FEG$, $\tri AFB$ and $\tri BGC$ are all
isosceles; and $B$ bisects $FG$. Therefore,
\begin{align*}
EF + EG > FG  &\implies  EF > CG \mot{or} AF
\\
&\implies  AE = AF + EF < 2EF.
\end{align*}

We claim that
\begin{equation}
(\tri FEG) > \frac{1}{4} (\tri AEC).
\label{eq:HH-step-3} 
\end{equation}
To see that, draw the straight line $EB$ cutting $AC$ perpendicularly
at~$D$ (Figure~\ref{fg:Heron-Huygens-2}). Then
$$
EF > \frac{1}{2} AE  \implies  BE > \frac{1}{2} DE
$$
since these lines are in the same proportion, because
$FB \parallel AD$. Also, $\tri FEG \sim \tri AEC$, so
$$
\frac{1}{2} EF \. BE > \frac{1}{8} AE \. DE
\implies \frac{1}{2} FG \cdot BE > \frac{1}{8} AC \cdot DE
$$
which confirms \eqref{eq:HH-step-3}.

\goodbreak 

Moreover,
$$
\frac{AF}{AE} = \frac{\text{altitude of $\tri ABC$}}
{\text{altitude of $\tri AEC$}} = \frac{BD}{DE}
$$
and the two triangles have the same base $AC$. Hence,
$$
AF < \frac{1}{2} AE  \implies  BD < \frac{1}{2} DE
\implies  (\tri ABC) < \frac{1}{2}(\tri AEC).
$$
This last inequality, combined with \eqref{eq:HH-step-3}, yields the
desired result:
$$
(\tri FEG) > \frac{1}{2} (\tri ABC).
\eqno \qed 
$$
\hideqed
\end{proof}

\begin{thm}[Heron--Huygens II: Theorem  IV of \emph{Inventa}] 
\label{th:Heron-Huygens-2}
The area of a circular segment less than a semicircle is \textbf{less}
than \textbf{two thirds} of the area of a triangle having its base in
common with this segment and its sides \textbf{tangent} to it.
\end{thm}

\begin{proof}
Referring again to Figure~\ref{fg:Heron-Huygens-2}, we obtain:
$$
(\tri FEG) > \frac{1}{2} (\tri ABC), \quad
(\tri HFK) > \frac{1}{2} (\tri AJB), \quad\text{and so on}.
$$
That is, we repeat the process with smaller and smaller triangles; and
we obtain an infinite sequence of inequalities in which the area on
the left-hand side is greater that one-half of the area on the
right-hand side.

Adding all these inequalities, we obtain
$$
(\tri AEC) - (\segment ABC) > \frac{1}{2} (\segment ABC).
$$
In other words,
$$
(\segment ABC) < \frac{2}{3} (\tri AEC),
$$
which establishes the theorem.
\end{proof}


\subsection{The second Snell theorem} 
\label{ssc:Snell-two}

\begin{lema}[Theorem VI of \emph{Inventa}] 
\label{lm:area-estimate}
If $A_n$ is the area of an \emph{inscribed} regular polygon of
$n$~sides, $S$ the area of the circle, and $A'_n$ the area of a
\emph{circumscribed} regular polygon of $n$~sides, then
\begin{equation}
\boxed{S < \frac{2}{3}\, A'_n + \frac{1}{3}\, A_n} \,.
\label{eq:area-estimate} 
\end{equation}
\end{lema}

\begin{figure}[htb]
\centering
\begin{tikzpicture}
\coordinate (O) at (0,0) ;
\coordinate (B) at (0:2cm) ; \coordinate (C) at (60:2cm) ;
\coordinate (B1) at (120:2cm) ; \coordinate (C1) at (180:2cm) ;
\coordinate (B2) at (-120:2cm) ; \coordinate (C2) at (-60:2cm) ;
\coordinate (D) at (30:2cm) ;
\coordinate (E) at (-30:2.309cm) ; \coordinate (F) at (30:2.309cm) ;
\coordinate (G) at (90:2.309cm) ; \coordinate (E1) at (150:2.309cm) ;
\coordinate (F1) at (-150:2.309cm) ;
\coordinate (G1) at (-90:2.309cm) ;
\draw (O) node[above left] {$O$} circle(2cm) ;
\draw (B) node[right=2pt] {$B$} -- (C) node[above right] {$C$} 
   -- (O) -- cycle ;
\draw[dashed] (C) -- (B1) -- (C1) -- (B2) -- (C2) -- (B) ;
\draw (E) node[below right] {$E$} -- (F) node[above right] {$F$} 
    -- (G) node[above right] {$G$} ;
\draw[dashed] (G) -- (E1) -- (F1) -- (G1) -- (E) ;
\draw (D) node[below right=-2pt] {$D$} ;
\foreach \pt in {B,B1,B2,C,C1,C2,D,E,E1,F,F1,G,G1,O}
    \draw (\pt) node{$\bull$} ;
\end{tikzpicture}
\caption{Inscribed and circumscribed regular polygons}
\label{fg:Huygens-VI} 
\end{figure}

\begin{proof}
Given a circle with center $O$; let there be inscribed in it an
equilateral polygon%
\footnote{%
An equilateral polygon inscribed in a circle is actually a
\textit{regular} polygon.}
(say a hexagon), one of whose sides is $BC$; and let $EFG\cdots$ be
circumscribed and similar to it, with the sides tangent to the circle
at the vertices of the first polygon: see Figure~\ref{fg:Huygens-VI}.

We claim that the area of the circle is less than two thirds the area
of polygon $EFG\cdots$ plus one third of that of polygon~$BC\cdots$,
that is,
\begin{equation}
S < \frac{2}{3}(\polygon EFG\cdots) + \frac{1}{3}(\polygon BC\cdots).
\label{eq:area-estimate-bis} 
\end{equation}

To see that, draw the radii $OB$ and $OC$. Then, since $\tri BFC$
rests on the base $BC$ of the segment~$BDC$, with its other sides
tangent to the segment $BDC$, Theorem~\ref{th:Heron-Huygens-2}
(Theorem~IV of \emph{Inventa}) shows that
$$
(\segment BDC) < \frac{2}{3} (\tri BFC).
$$
It follows that
$$
(\tri OBC) + \frac{2}{3} (\tri BFC) > (\sector OBC).
$$
In other words,
\begin{align*}
(\tri OBC) 
+ \frac{2}{3} \bigl[ (\quadrangle OBFC) - (\tri OBC) \bigr]
&> (\sector OBC),
\\
\text{or} \quad
\frac{2}{3} (\quadrangle OBFC) + \frac{1}{3} (\tri OBC)
&> (\sector OBC).
\end{align*}
The area estimate \eqref{eq:area-estimate-bis} is then obtained by
taking the sum as many times as copies of the $\sector OBC$ are
contained in the circle. The theorem follows, since
\eqref{eq:area-estimate-bis} is just a restatement of the estimate
\eqref{eq:area-estimate}.
\end{proof}

\begin{thm}[Theorem VIII of \emph{Inventa}] 
\label{th:Huygens-VIII}
Given a circle, if at the extremity of a diameter a tangent is drawn,
and if from the opposite extremity of the diameter a line is drawn
which cuts the circumference and meets the tangent produced, then
\textbf{two thirds} of the intercepted tangent plus \textbf{one third}
of the line dropped from the point of intersection perpendicular to
the diameter are \textbf{greater} than the adjacent \textbf{subtended
arc}.
\end{thm}

\begin{figure}[htb]
\centering
\begin{tikzpicture}
\coordinate (O) at (0,0) ;
\coordinate (B) at (-2,0) ;
\coordinate (C) at (2,0) ;
\coordinate (E) at (70:2cm) ;
\coordinate (F) at ($ (B)!(E)!(C) $) ;
\coordinate (Ld) at (2,-1.5) ; \coordinate (Lu) at (2,3.2) ;
\coordinate (D) at (intersection of B--E and C--Lu) ;
\draw (E) ++ (-20:1cm) coordinate (Eg) ;
\coordinate (G) at (intersection of E--Eg and C--Lu) ;
\draw (O) circle(2cm) ;
\draw ($ (E)!-0.5!(G) $) -- ($ (E)!1.5!(G) $) ;
\draw (B)  node[below left] {$B$} -- (C) node[right] {$C$}
   -- (G) node[above right] {$G$} -- (D) node[right] {$D$} -- cycle ;
\draw (Ld) -- (Lu)  ;
\draw (E) node[above=2pt] {$E$} -- (F) node[below] {$F$} ;
\draw (O) node[below left] {$O$} -- (E) -- (C) ;
\draw[dashed] (O) -- (G) ;
\draw[gray] (C) ++(-0.2,0) -- ++(0,0.2) -- ++(0.2,0) ;      
\draw[gray] (E) ++(-145:2mm) -- ++(-55:2mm) -- ++(35:2mm) ; 
\draw[gray] (F) ++(0.2,0) -- ++(0,0.2) -- ++ (-0.2,0) ;     
\foreach \pt in {B,C,D,E,F,G,O} \draw (\pt) node{$\bull$} ;
\end{tikzpicture}
\caption{Another arc-length comparison}
\label{fg:Huygens-VIII} 
\end{figure}

\begin{proof}
Take a circle with center $O$ and diameter $BC$; and draw from $C$ a
line $CD$ tangent to the circle. And let a line $BD$, drawn from the
other extremity of the diameter, meet this line at~$D$ and intersect
the circumference at $E$; let $EF$ be the perpendicular from~$E$ to
the diameter $BC$ (see Figure~\ref{fg:Huygens-VIII}). One must prove
that
\begin{equation}
\boxed{\arc{CE} < \frac{2}{3}\,CD + \frac{1}{3}\,EF}\,.
\label{eq:Huygens-VIII} 
\end{equation}
Join $OE$ and $CE$. At the point $E$ draw a tangent to the
circle which meets the tangent $CD$ at~$G$. Then
$$
EG = CG = DG.
$$
Indeed, if we draw a circle with center $G$ which passes through the
points $C$ and~$E$, it will also pass through the point $D$ because
$\angle CED$ is a right angle. By the proof of
Lemma~\ref{lm:area-estimate},
\begin{align*}
(\sector OEC) < \frac{2}{3} (\quadrangle OEGC)
+ \frac{1}{3} (\tri EOC).
\end{align*} 
Now, $(\quadrangle OEGC) = 2(\tri OCG)$ equals the area of a triangle
with base $2CG = CD$ and an altitude $OC$; whereas $(\tri OEC)$ equals
the area of a triangle with base $EF$ and the same altitude $OC$.
Therefore,
\begin{align*}
\frac{2}{3} (\quadrangle OEGC) + \frac{1}{3} (\tri EOC)
&= \biggl( \tri \text{ with base $\frac{2}{3} CD + \frac{1}{3} EF$ 
and altitude $OC$} \biggr)
\\
&= \frac{1}{2} \biggl( \frac{2}{3} CD + \frac{1}{3} EF \biggr) \. OC
\\
& > (\sector OEC) = \frac{1}{2} \arc{CE} \. OC,
\end{align*}
which entails the required inequality \eqref{eq:Huygens-VIII}.
\end{proof}

\begin{thm}[Second Snell Theorem: Theorem IX of \emph{Inventa}] 
\label{th:Huygens-IX}
The circumference of a circle is less than \textbf{two thirds} of the
perimeter of an equilateral polygon \textbf{inscribed} in it plus
\textbf{one third} of a similar \textbf{circumscribed} polygon.
\end{thm}

\begin{figure}[htb]
\centering
\begin{tikzpicture}
\coordinate (O) at (0,0) ;
\coordinate (B) at (-150:2cm) ;
\coordinate (C) at (0:2cm) ; \coordinate (D) at (60:2cm) ;
\coordinate (C1) at (120:2cm) ; \coordinate (D1) at (180:2cm) ;
\coordinate (C2) at (-120:2cm) ; \coordinate (D2) at (-60:2cm) ;
\coordinate (G) at (30:2cm) ;
\coordinate (E) at (0:2.309cm) ; \coordinate (F) at (60:2.309cm) ;
\coordinate (E1) at (120:2.309cm) ; \coordinate (F1) at (180:2.309cm) ;
\coordinate (E2) at (-120:2.309cm) ; \coordinate (F2) at (-60:2.309cm) ;
\coordinate (H) at (intersection of C--D and O--G) ;
\coordinate (K) at (intersection of B--C and E--F) ;
\coordinate (L) at ($ (H)!-1.0!(G) $) ;
\draw (O) node[above left] {$O$} circle(2cm) ;
\draw (C) node[below left] {$C$} -- (D) node[below left] {$D$} ;
\draw[dashed] (D) -- (C1) -- (D1) -- (C2) -- (D2) -- (C) ;
\draw (E) node[below right] {$E$} -- (F) node[above right] {$F$} ;
\draw[dashed] (F) -- (E1) -- (F1) -- (E2) -- (F2) -- (E) ;
\draw (O) -- (E) ;
\draw (B) node[below left] {$B$} -- (L) node[above left] {$L$}
   -- (H) node[below] {$H$} -- (G) node[right] {$G$} ;
\draw (B) -- (K) node[above right] {$K$} ;
\foreach \pt in {B,C,C1,C2,D,D1,D2,E,E1,E2,F,F1,F2,G,H,K,L,O}
    \draw (\pt) node{$\bull$} ;
\end{tikzpicture}
\caption{Circumference versus polygon perimeters}
\label{fg:Huygens-IX} 
\end{figure}

\begin{proof}
In symbols, we must prove that
\begin{equation}
\boxed{C < \frac{2}{3}\, C_n + \frac{1}{3}\, C'_n}
\label{eq:Huygens-IX} 
\end{equation}
where $C_n$ denotes the perimeter of the inscribed polygon of
$n$~sides, $C_n'$ is the perimeter of the circumscribed polygon of
$n$~sides, and $C$ is the circumference of the circle.

Given a circle with center $O$, let there be inscribed in it an
equilateral polygon, one of whose sides is $CD$. Let there be
circumscribed another polygon with sides parallel to the former; call
$EF$ its side which is parallel to~$CD$. (See
Figure~\ref{fg:Huygens-IX}.) \textit{We have to prove that the
circumference of the circle is less than two thirds of the perimeter
of polygon $CD\cdots$ plus one third of the perimeter of
polygon~$EF\cdots$}.

Draw the diameter $BG$ of the circle that bisects the side $CD$ of the
inscribed polygon at~$H$, and the parallel side $EF$ of the
circumscribed polygon at~$G$ (it is obvious that $G$ is the point of
tangency of the side $EF$). Choose $L$ on the diameter $BG$ between
$O$ and $H$ such that $HL = GH$; draw $BC$, produced to meet the side
$EF$ at~$K$, and extend the radius $OC$ to meet the vertex $E$ of the
circumscribed polygon. Then
$$
HL = GH \implies  BL = 2OH \implies  \frac{OG}{OH} = \frac{BG}{BL}\,.
$$
But now
$$
\frac{BH}{BL} > \frac{BG}{BH}
$$
since $BG > BH > BL$, exceeding one another by the same amount. 
Therefore,
$$
\frac{OG}{OH} = \frac{BG}{BL}
= \frac{BG}{BH}\, \frac{BH}{BL} > \frac{BG^2}{BH^2}\,.
$$
Consequently, by similar triangles,
$$
\frac{OG}{OH} = \frac{EG}{CH}
\word{and} \frac{BG}{BH} = \frac{KG}{CH}
\implies  \frac{EG}{CH} > \frac{KG^2}{CH^2} \,.
$$
This in turn implies that
$$
\frac{EG}{KG} > \frac{KG}{CH}\,.
$$
Since the arithmetic mean of $EG$ and $CH$ is greater than their
geometric mean, we infer
$$
EG + CH > 2\sqrt{EG \. CH} > 2\sqrt{KG^2} = 2\,KG.
$$
So, in particular, $\third(EG + CH) > \twothirds KG$.

Hence, finally,
\begin{equation}
\frac{1}{3} EG + \frac{2}{3} CH
> \frac{2}{3} KG + \frac{1}{3} CH > \arc{CG},
\label{eq:Huygens-IX-arc} 
\end{equation}
where the last inequality follows from \eqref{eq:Huygens-VIII} in the
previous Theorem~\ref{th:Huygens-VIII}. 

And now one extends the inequality \eqref{eq:Huygens-IX-arc} for the
arc $\arc{CG}$ to the entire circle.
\end{proof}

\paragraph{Error Analysis}
If $p'_n$ is the perimeter of a circumscribed regular $n$-gon with
central angle $\frac{2\pi}{n}$ in a circle of radius~$1$, then
$$
p'_n = 2n \tan\Bigl( \frac{\pi}{n} \Bigr) \word{and}
p_n = 2n \sin\Bigl( \frac{\pi}{n} \Bigr),
$$
and we obtain
$$
\frac{2}{3} p_n + \frac{1}{3} p'_n
= 2\pi + \frac{\pi^5}{10n^4} + \frac{\pi^7}{8n^6} +\cdots
$$
which shows the approximation to be in \textit{excess}; and the
\textit{error} is $\dfrac{\pi^5}{10n^4} \doteq \dfrac{30.6}{n^4}$
which is small for large~$n$. One also sees that it is somewhat less
accurate than the lower bound, although it is of the same order.


\subsection{Snell's upper bound} 
\label{ssc:Snell-upper-bound}

As we pointed out earlier, Huygens proves two famous inequalities. One
of them is the Cusa inequality of a previous section. The second one
we discuss now. It is fascinating that the figure involved is the
\textit{Archimedes trisection figure} -- see, for instance,
\cite{Heath1912} and the last section of this paper.

\begin{thm}[Snell's Upper Bound: Theorem XII of \emph{Inventa}] 
\label{th:Huygens-XII}
If between the diameter produced of a circle and its circumference a
line equal to the radius is inserted, which when produced cuts the
circle and meets a tangent to the circle at the other extremity of the
diameter, this line will intercept a part of the tangent
\textbf{greater} than the adjacent intercepted \textbf{arc}.
\end{thm} 

\begin{figure}[htb]
\centering
\begin{tikzpicture}
\coordinate (O) at (0,0) ;
\coordinate (A) at (-2,0) ; 
\coordinate (B) at (2,0) ; 
\coordinate (E) at (-3.8,0) ; 
\coordinate (G) at (2,2) ; 
\coordinate (Oh) at (-5.8,-2) ;
\coordinate (Om) at (5.8,2) ; 
\coordinate (L) at (intersection of B--G and Oh--Om) ;
\draw[name path=bigOh] (O) circle(2cm) ;
\path[name path=Eg] (E) -- (G) ;
\path[name intersections={of=Eg and bigOh, by={F,D}}] ;
\path[name path=Hl] (Oh) -- (Om) ;
\path[name intersections={of=Hl and bigOh, by={M,H}}] ;
\coordinate (K) at (intersection of A--B and D--H) ;
\draw (B) node[below right] {$B$} -- (G) node[right] {$G$} ;
\draw (E) node[below left] {$E$} -- (A) node[below left] {$A$}
   -- (O) node[below=2pt] {$O$} -- (B) ;
\draw (E) node[below left] {$E$} -- (D) node[above left] {$D$}
   -- (F) node[above] {$F$} -- (G) ;
\draw (H) node[below left] {$H$} 
   -- (M) node[above left=-2pt] {$M$} -- (L) node[right] {$L$} ;
\draw (D) -- (K) node[above right=-2pt] {$K$} -- (H) ;
\foreach \mk in {A,B} \draw (\mk) node{$\sst\circ$} ;
\foreach \pt in {D,E,F,G,H,L,K,M,O}
    \draw (\pt) node{$\bull$} ;
\end{tikzpicture}
\caption{Arc and tangent intercepts}
\label{fg:Huygens-XII} 
\end{figure}

\begin{proof}
Given a circle with center $O$ and diameter $AB$, produce the diameter
beyond~$A$ to a point~$E$, and take the point~$D$ on the circle such
that the line $ED$ is equal to the radius~$OA$. When produced, the
line $ED$ cuts the circumference again at~$F$ and meets at the
point~$G$ the tangent to the circle at~$B$. (See
Figure~\ref{fg:Huygens-XII}.) We shall prove that
\begin{equation}
\boxed{BG > \arc{BF}}\,.
\label{eq:Snells-own} 
\end{equation}
Draw $HL$ through the center~$O$ parallel to $EG$, meeting the
circumference at $H$ and~$M$ and the tangent $BG$ at~$L$. Draw the
line $DH$, cutting the diameter $AB$ at~$K$. Then
$$
\tri EDK \sim \tri OHK
$$
since the angles at $K$ are equal and $\angle DEK = \angle HOK$. But
$ED = OH$, and these sides are subtended by equal angles; hence,
$DK = HK$.

Therefore, the radius $OA$ bisects both the chord $DH$
(perpendicularly) and the arc $\arc{DAH}$. Since $HM \parallel DF$ by
construction, it follows that
$$
\arc{FM} = \arc{DH} = 2\,\arc{AH}.
$$

But $\arc{AH} = \arc{BM}$, and therefore
$$
\arc{BF} = \arc{BM} + \arc{FM} = 3\,\arc{AH}.
$$

Moreover, assuming that the radius $OA$ has length~$1$,
$$
HK = \text{sine of } \arc{AH} \word{and}
BL = \text{tangent of } \arc{AH}.
$$
These relations, together with formula \eqref{eq:Huygens-IX} --
or~\eqref{eq:Huygens-IX-arc} -- of the second Snell theorem, imply
that
$$
\frac{2}{3} HK + \frac{1}{3} BL > \arc{AH}.
$$
Noting that $DHLG$ is a parallelogram, we deduce the required
inequality \eqref{eq:Snells-own}:
\begin{align*}
2\,HK + BL > 3\,\arc{AH}
& \implies BL + DH > 3\,\arc{AH}
\\
& \implies BL + LG > 3\,\arc{AH} \implies BG > \arc{BF}.
\tag*{\qed} 
\end{align*}
\hideqed
\end{proof}

\paragraph{Error Analysis}
If we now put $x := \arc{AH}$, so that $3x = \arc{BF}$, then
$$
\begin{aligned}
LG = DH &= 2\sin x  \\
BL      &=  \tan x  \end{aligned} \Biggr\}
\implies BG = 2\sin x + \tan x.
$$
The estimate $\third \arc{BF} \doteq \third BG$ is Snell's own
approximation:
$$
x \doteq \frac{2\sin x + \tan x}{3} \,.
$$

The expansion
$$
\frac{2\sin x + \tan x}{3} 
= x + \frac{x^5}{20} + \frac{x^7}{56} +\cdots
$$
shows that the approximation is in \textit{excess}; and the error one
commits in the approximation is about $-x^5/20$.

If we apply the \textit{first} (Cusa) inequality
\eqref{eq:Snell-second} for $n = 6$ and this \textit{second} (Snell)
inequality \eqref{eq:Huygens-IX} for $n = 12$, we obtain
$$
3.1411\cdots < \pi < 3.1424\cdots
$$
which only give \emph{two}-place accuracy. Archimedes needed a
$96$-gon to obtain similar bounds. This shows the extraordinary
improvement that the Snell--Cusa inequalities achieve over Archimedes'
original computation.

If instead we apply the first (Cusa) inequality for $n = 30$ and the
second (Snell) inequality for $n = 60$, we arrive at
$$
3.1415917\cdots < \pi < 3.141594\cdots
$$
which shows that the decimal expansion of $\pi$ begins with
$\pi \doteq 3.14159$. Even with this \emph{five}-place accuracy, we
still don't need the full $96$~sides that Archimedes used.


\section{Huygens' barycentric theorems} 
\label{sec:barycenter}

\subsection{The barycenter} 
\label{ssc:barycenter}

After proving the Cusa--Snell inequalities, Huygens offers his own
very elegant approximation which is based on an observation about the
barycenter of a circular segment. What is novel and original is how
Huygens \textit{transforms the location of a barycenter into an
inequality about perimeters}.

\begin{thm}[Theorem XIV of \emph{Inventa}] 
\label{th:Huygens-XIV}
The barycenter of a circular segment divides the diameter of the
segment so that the part near the vertex is \textbf{greater} than the
rest, but \textbf{less} than \textbf{three halves} of it.
\end{thm}

\begin{figure}[htb]
\centering
\begin{tikzpicture}
\coordinate (A) at (10:2cm) ;
\coordinate (B) at (90:2cm) ;
\coordinate (C) at (170:2cm) ;
\coordinate (D) at ($ (A)!0.5!(C) $) ;
\coordinate (E) at ($ (B)!0.5!(D) $) ;
\draw (E) ++ (2,0) coordinate (Ef) ;
\draw (E) ++ (-2,0) coordinate (Eg) ;
\path[name path=bigOh] (A) arc(10:170:2cm) ;
\path[name path=Fg] (Ef) -- (Eg) ;
\path[name intersections={of=Fg and bigOh, by={F,G}}] ;
\coordinate (L) at ($ (A)!(G)!(C) $) ;
\coordinate (J) at ($ (A)!(F)!(C) $) ;
\coordinate (K) at ($ (J)!2.0!(F) $) ;
\coordinate (H) at ($ (L)!2.0!(G) $) ;
\draw (A) arc(10:170:2cm) -- cycle ;  
\draw (A) node[right] {$A$} -- (C) node[left] {$C$} ;
\draw (B) node[above=2pt] {$B$} -- (D) node[below=2pt] {$D$} ;
\draw (F) node[right=2pt] {$F$} -- (E) node[above right] {$E$}
   -- (G) node[left=2pt] {$G$} ;
\draw[dashed] (J) node[below=2pt] {$J$}
   -- (K) node[above right] {$K$} -- (H) node[above left] {$H$}
   -- (L) node[below=2pt] {$L$} ;
\foreach \pt in {A,B,C,D,E,F,G,H,J,K,L}
   \draw (\pt) node{$\bull$} ;
\end{tikzpicture}
\caption{A rectangle framing a circular segment}
\label{fg:Huygens-XIV-one} 
\end{figure}

\begin{proof}
Take a segment $ABC$ of a circle (and let it be put less than a
semicircle because others do not satisfy the proposition), and let
$E$ be the midpoint of its diameter~$BD$.

\paragraph{Step 1}
First, we show that the barycenter of the segment $ABC$ lies
\textit{below} this midpoint $E$ (viewed from the vertex~$B$).

It is evident from considerations of symmetry that the barycenter lies
on the diameter~$BD$.

Draw a line through $E$ parallel to the base $AC$, meeting the
circumference on either side at points $F$ and~$G$. Draw the lines
$KJ$ through~$F$ and $HL$ through~$G$, perpendicular to~$AC$; these,
together with the line $KH$ tangent to the segment at its vertex~$B$,
form the rectangle~$KHLJ$: see Figure~\ref{fg:Huygens-XIV-one}. Since,
by assumption, the segment $ABC$ is less than a semicircle, the
rectangle $FGLJ$, which is one half of the given rectangle, is
contained \textit{within} the segment $AFGC$; and the regions $AFJ$
and $LGC$ are left over.

But $KHGF$, the other half of the rectangle $KHLJ$, \textit{includes}
the segment $FBG$; and also includes the regions $FBK$ and $GBH$.
Since those two regions lie wholly \textit{above} the line $FG$, the
\emph{barycenter of their union} will be located above~it.

Now the point $E$, on this same line $FG$, is the barycenter of the
whole rectangle $KHLJ$. Therefore the barycenter of the remaining region
$BFJLGB$ will lie \textit{below} the line~$FG$.
  
But the barycenter of the pair of regions $AFJ$ and $LGC$ lies also
\textit{below} $FG$. Therefore, the barycenter of the magnitude
composed of these regions and the region $BFJLGB$, i.e., of the whole
segment $ABC$, must also be found below the line $FG$, and hence
\textit{below the point}~$E$ on the diameter.

\begin{figure}[htb]
\centering
\begin{tikzpicture}
\coordinate (O) at (0,0) ;
\coordinate (A) at (10:2cm) ;
\coordinate (B) at (90:2cm) ;
\coordinate (C) at (170:2cm) ;
\coordinate (D) at ($ (A)!0.5!(C) $) ;
\draw (A) ++ (0.25,0) coordinate (H) ;  
\draw (C) ++ (-0.25,0) coordinate (K) ;
\coordinate (L) at ($ (B)!0.3!(D) $) ;
\draw (L) ++ (2,0) coordinate (Ln) ;
\draw (L) ++ (-2,0) coordinate (Lq) ;
\coordinate (P) at (-90:2cm) ;
\coordinate (S) at ($ (B)!0.6!(D) $) ;
\draw (S) ++ (2,0) coordinate (Sf) ;
\draw (S) ++ (-2,0) coordinate (Sg) ;
\path[name path=bigOh] (O) circle(2cm) ;
\path[name path=Ug] (Sf) -- (Sg) ;
\path[name path=Ln] (Lq) -- (Ln) ;
\path[name intersections={of=Ug and bigOh, by={U,V}}] ;
\path[name intersections={of=Ln and bigOh, by={N,Np}}] ;
\draw (O) circle (2cm) ;
\draw (H) node[right, blue] {$H$} -- (A) node[below left] {$A$}
   -- (C) node[below right] {$C$} -- (K) node[left, blue] {$K$} ;
\draw (B) node[above] {$B$} -- (S) node[below left=-2pt] {$S$}
   -- (D) node[below left] {$D$} ;  
\draw (U) node[right] {$U$} -- (V) node[left] {$V$} ;
\draw[dashed] (D) -- (P) node[above right] {$P$} ;
\draw[blue, name path=Para] (H)  parabola bend (B)  (K)  ;
\path[name intersections={of=Ln and Para, by={M,Q}}] ;
\draw[dashed] (Q) node[above left, blue] {$Q$} -- (L) node[left] {$L$}
   -- (M) node[above=2pt, blue] {$M$} -- (N) node[right] {$N$} ;
\foreach \pt in {A,B,C,D,L,P,N,S,U,V}
   \draw (\pt) node{$\bull$} ;
\foreach \mk in {H,K,M,Q}  \draw[blue] (\mk) node{$\sst\circ$} ;
\end{tikzpicture}
\caption{A parabolic segment framing a circular segment}
\label{fg:Huygens-XIV-two} 
\end{figure}

\paragraph{Step 2}
Now cut the same diameter $BD$ at~$S$ so that $BS$ is \textit{three
halves} of the remainder~$SD$. We assert that the barycenter of the
segment $ABC$ is closer to~$B$ than the point~$S$.

Let $BP$ be the diameter of the whole circle, as in
Figure~\ref{fg:Heron-Huygens-1}. Draw a line through $S$ parallel to
the base $AC$, meeting the circumference in $U$ and~$V$. Next, draw a
\textit{parabola} with vertex~$B$, axis $BD$, and \textit{latus
rectum} equal to~$SP$; meeting the base $AC$ of the circular segment
at $H$ and~$K$. Then, since
$$
US^2 = VS^2 = BS \. SP,
$$
the parabola will cross the circle at the points $U$ and~$V$. (See
Figure~\ref{fg:Huygens-XIV-two}.)

Now, the arcs $\arc{BU}$ and $\arc{BV}$ of the parabola will fall
\textit{within} the circle, but the remaining arcs $\arc{UH}$ and
$\arc{VK}$ will lie \textit{outside}~it. We prove this as follows:
through any point $L$ on the diameter between $B$ and~$S$, draw the
line $NL$ parallel to the base $AC$, meeting the circumference of the
circle at~$N$ and the parabola at points $M$ and~$Q$. Now the
relations
$$
\begin{aligned}
NL^2 &= BL \. BP \\
ML^2 &= BL \. SP 
\end{aligned} \,\Biggr\} 
\word{together with}  BL \. BP > BL\. SP
$$
imply that $NL^2 > ML^2$, and hence $NL > ML = QL$.

The same holds for any such parallel line drawn between $B$ and~$S$;
and therefore the arcs $\arc{BU}$ and $\arc{BV}$ of the circumference
must lie entirely outside of the parabola.

Again, since
$$
\begin{aligned}
AD^2 &= BD \. DP \\
HD^2 &= BD \. SP \end{aligned} \,\Biggr\}  \implies  HD > AD,
$$
and the same will hold for any line parallel to the base drawn between
$S$ and~$D$. Therefore the arcs $\arc{UA}$ and $\arc{VC}$ will fall
\textit{within} the parabola.

Now we examine the regions $UNBM$, and $BQV$, as well as $HUA$ and
$VCK$. Since the latter two lie entirely below the line $UV$, the
barycenter of their union will also lie below it. But the barycenter
of the parabolic segment $HBK$ lies on $UV$ at the point~$S$.%
\footnote{%
{Archimedes, \textit{Equilibrium of Planes}, Book~II, Proposition~8
\cite{Heath1912}.
This theorem states that if $AO$ is the diameter of a parabolic
segment with vertex~$A$, and if $G$ is its barycenter, then 
$AG = \frac{3}{2}\,GO$.}}
Therefore, the barycenter of the remaining region $AUMBQVC$ will lie
\textit{above} the line $UV$. But the barycenters of the regions
$UMBN$ and $BVQ$ that also are above the line $UV$ will likewise lie
above that line. Therefore, the region composed of these two together
with $AUMBQVC$, i.e., the segment $ABC$ of the circle, will have its
barycenter above the line~$UV$. And since that barycenter lies on the
diameter $BD$, it will be \textit{nearer} the vertex $B$ than the
point~$S$.
\end{proof}

We now prove a theorem, also due to Huygens in an earlier work
\emph{Theoremata}~\cite{Huygens1651}, which relates the area of a
circular segment to the position of its barycenter.

\begin{thm}[Theorem VII of \emph{Theoremata}] 
\label{th:Huygens-Theoremata}
The area of a circular segment is to the area of the inscribed
triangle with the same base and height as two-thirds of the diameter
of the opposite segment is to the distance from the center of the
circle to the barycenter of the [original] segment.
\end{thm}

\begin{figure}[htb]
\centering
\begin{tikzpicture}[rotate=-30]
\coordinate (O) at (0,0) ;
\coordinate (A) at (30:2cm) ;
\coordinate (B) at (90:2cm) ;
\coordinate (C) at (150:2cm) ;
\coordinate (D) at ($ (A)!0.5!(C) $) ;
\coordinate (P) at (-90:2cm) ;
\coordinate (G) at (0,-1.732) ;
\coordinate (H) at (-1.732,-1.732) ;
\coordinate (K) at (1.732,-1.732) ;
\coordinate (L) at ($ (B)!0.54!(D) $) ;  
\coordinate (M) at (0,-1.155) ;
\draw (O) circle(2cm) ; 
\draw (A) node[right] {$A$} -- (B) node[above right] {$B$}
    -- (C) node[above left] {$C$} -- cycle ;
\draw (K) node[right] {$K$} -- (O) node[right] {$O$}
    -- (H) node[above left] {$H$} -- cycle ;
\draw[dashed] (B) -- (L) node[left] {$L$} 
    -- (D) node[below left=-2pt] {$D$}  -- (M) node[right] {$M$}
    -- (G) node[above=2pt] {$G$} -- (P) node[below left] {$P$} ;
\draw[gray] (D) ++(0:2mm) -- ++(90:2mm) -- ++ (180:2mm) ;
\draw[gray] (G) ++(0:2mm) -- ++(90:2mm) -- ++ (180:2mm) ;
\foreach \pt in {A,B,C,D,G,H,K,P,O}  \draw (\pt) node{$\bull$} ;
\foreach \mk in {L,M}  \draw[red] (\mk) node{$\bull$} ;
\end{tikzpicture}
\caption{Balancing a circular segment with an isosceles triangle}
\label{fg:Huygens-Theoremata} 
\end{figure}

\begin{proof}
Let $ABC$ and $\tri ABC$ be the given segment and triangle. Let $BD$
be the diameter of the segment and prolong it to the center $O$ of the
circle. Let $PD$ be the diameter of the remaining segment. Let $L$ be
the barycenter of $\segment{ABC}$. The theorem may then be restated
as:
\begin{equation}
\label{eq:balancing-act} 
\frac{(\segment{ABC})}{(\tri ABC)}
= \frac{2}{3}\, \frac{PD}{OL} \,.
\end{equation}

To prove \eqref{eq:balancing-act}, locate $G$ on the diameter $BP$ of the
circle between $O$ and $P$ such that
\begin{equation}
OG^2 = BD \. PD
\label{eq:counter-weight} 
\end{equation}
Draw the line $KH$ through $G$ parallel to the base $AC$ of the first
segment, such that $KH = AC$ and $G$ is the midpoint of~$KH$. Let $M$
be the barycenter of the triangle $\tri KOH$.

We now need the following auxiliary proposition, also taken from
\emph{Theoremata} \cite{Huygens1651}. We defer its long (and very
Archimedean) proof to Appendix~\ref{app:helpful}.

\begin{prop}[Theorem V of \emph{Theoremata}] 
\label{pr:Huygens-Theoremata}
The barycenter of the figure that combines the segment $ABC$ of the
circle with the triangle $\tri KOH$ just described lies at the centre
of the circle.
\qned
\end{prop}

Now, by a famous theorem of Archimedes,%
\footnote{%
Archimedes, \textit{Equilibrium of Planes}, Proposition~I.14
\cite[p.~201]{Heath1912}: the barycenter of a triangle lies at the 
intersection of the medians. (The ratio $2{:}1$ on each median 
follows easily.)}
\begin{equation}
OM = \frac{2}{3} OG.
\label{eq:barycenter-KOH} 
\end{equation}
Thus,
$$
\frac{(\tri KOH)}{(\tri ABC)} = \frac{OG}{BD} = \frac{PD}{OG} \,,
$$
using~\eqref{eq:counter-weight}. Combined with 
\eqref{eq:barycenter-KOH}, this yields
$$
\frac{(\tri KOH)}{(\tri ABC)} = \frac{\frac{2}{3}\,PD}{OM} \,.
$$
However, the aforementioned barycenters are at $M$ and~$L$, so, by
Proposition~\ref{pr:Huygens-Theoremata}, they balance at~$O$.
Therefore,
$$
\frac{(\segment{ABC})}{(\tri KOH)} = \frac{OM}{OL}
$$
and, on multiplying the previous two ratios, \eqref{eq:balancing-act}
follows at once.
\end{proof}

Now we transform this result on the barycenter of a circular segment
into an inequality on its \textit{area}.

\begin{thm}[Theorem XV of the \emph{Inventa}] 
\label{th:Huygens-XV}
The area of a circular segment less than a semicircle has a
\emph{greater} ratio to its maximum inscribed triangle than $4{:}3$,
but \emph{less} than the ratio which $\frac{10}{3}$ of the diameter of
the remaining segment has to the diameter of the circle plus three
times the line which reaches from the center of the circle to the base
of the segment.
\end{thm}

\begin{figure}[htb]
\centering
\begin{tikzpicture}
\coordinate (O) at (0,0) ;
\coordinate (A) at (10:2cm) ;
\coordinate (B) at (90:2cm) ;
\coordinate (C) at (170:2cm) ;
\coordinate (D) at ($ (A)!0.5!(C) $) ;
\coordinate (P) at (-90:2cm) ;
\coordinate (H) at ($ (D)!0.667!(P) $) ;
\coordinate (L) at ($ (B)!0.52!(D) $) ;
\coordinate (S) at ($ (B)!0.6!(D) $) ;
\path[name path=bigOh] (O) circle(2cm) ;
\draw (O) node[below left=-2pt] {$O$} circle (2cm) ;
\draw (A) node[right=2pt] {$A$} -- (B) node[above] {$B$}
   -- (C) node[left=2pt] {$C$} -- cycle ;
\draw (B) -- (L) node[above left=-2pt] {$L$} 
   -- (S) node[right] {$S$} -- (D) node[above=5pt, right] {$D$} ;
\draw[dashed] (D) -- (H) node[right] {$H$} 
   -- (P) node[below=2pt] {$P$} ;
\foreach \pt in {A,B,C,D,H,O,P,S}  \draw (\pt) node{$\bull$} ;
\draw[red] (L) node{$\bull$} ;
\end{tikzpicture}
\caption{Comparing areas of a circular segment
and its inscribed triangle}
\label{fg:Huygens-XV} 
\end{figure}

\begin{proof}
Take a circular segment $ABC$ less than a semicircle in which is
inscribed the maximum triangle $\tri ABC$. Extend the diameter $BD$ of
the segment be to a diameter $BP$ of the circle, passing through its
center~$O$. (See Figure~\ref{fg:Huygens-XV}.)

We must first prove that
$$
\boxed{\frac{(\segment ABC)}{(\tri ABC)} > \frac{4}{3}}\,.
$$

Again let $L$ be the barycenter of the \textit{segment} $ABC$. Using
Theorem~\ref{th:Huygens-XIV}, we obtain
$$
BD < 2\,BL \word{and so} PD = PB - BD > 2\,OB - 2\,BL = 2\,OL,
$$
and since $BP = 2\,OB$, we deduce that
$$
\frac{BP}{PD} < \frac{OB}{OL}\,, \word{whereby} 
\frac{PD}{OL} > \frac{BP}{OB} = 2, 
$$
and therefore $PD > 2\,OL$.

Now let $PD$ be cut at~$H$ such that $HD = 2\,HP$. Then, since
$HD = \twothirds PD$, it follows that
$$
HD > \frac{4}{3}\,OL,
$$
and also, by \eqref{eq:balancing-act},
$$
\frac{HD}{OL} = \frac{(\segment ABC)}{(\tri ABC)}\,,
\word{whereby} \frac{(\segment ABC)}{(\tri ABC)} > \frac{4}{3}
$$
as required.

\medskip

Secondly, we must prove that
\begin{equation}
\boxed{\frac{(\segment ABC)}{(\tri ABC)}
< \frac{10}{3} \. \frac{PD}{BP + 3\,OD}}\,.
\label{eq:Huygens-XV} 
\end{equation}

Choose $S$ on the diameter $BD$ of the segment, as in the proof of
Theorem~\ref{th:Huygens-XIV}, so that
$$
BS = \frac{3}{2}\,SD.  
$$
By Theorem~\ref{th:Huygens-XIV}, $S$ falls between $L$ and $D$ since
$L$ is the barycenter of $\segment ABC$. This means that $OL > OS$.
Then, since
$$
\frac{HD}{OL} = \frac{(\segment ABC)}{(\tri ABC)}  \word{and}
OL > OS \implies \frac{HD}{OL} < \frac{HD}{OS} \,,
$$
it follows that
$$
\frac{(\segment ABC)}{(\tri ABC)} < \frac{HD}{OS} \,.
$$
Now, since $OS = OD + DS = OD + \frac{2}{5}\,BD$, so that
$$
5\,OS = 2\,BD + 5\,OD = 2\,OB + 3\,OD = BP + 3\,OD,
$$
we obtain the desired estimate \eqref{eq:Huygens-XV}:
$$
\frac{(\segment ABC)}{(\tri ABC)} < \frac{5\,HD}{5\,OS}
= \frac{10}{3}\, \frac{PD}{5\,OS}
= \frac{10}{3}\, \frac{PD}{BP + 3OD} \,.
\eqno \qed
$$
\hideqed
\end{proof}

Now we are ready to apply the previous inequality on \textit{areas} to
obtain an inequality on a \textit{circular arc}.

\begin{thm}[Theorem XVI of \emph{Inventa}] 
\label{th:Huygens-XVI}
Any \emph{arc} less than a semicircumference is \emph{greater} than
its chord plus one-third of the difference between that chord and its
sine. But it is \emph{less} than the chord plus the line which has the
same ratio to the aforementioned one third that four times the chord
plus the sine has to twice the chord plus three times the sine.
\end{thm}

\begin{figure}[htb]
\centering
\begin{tikzpicture}
\coordinate (O) at (0,0) ;
\coordinate (A) at (0:2cm) ;
\coordinate (B) at (60:2cm) ;
\coordinate (C) at (120:2cm) ;
\coordinate (D) at ($ (A)!0.5!(C) $) ;
\coordinate (P) at (-120:2cm) ;
\coordinate (R) at (-60:2cm) ;
\coordinate (M) at ($ (C)!(A)!(R) $) ;
\begin{scope}[xshift=6cm, rotate=30]
\coordinate (G) at (-1.732,0) ;
\coordinate (H) at (0,0) ;
\coordinate (I) at (1.732,0) ;
\coordinate (J) at (0.5,0) ;
\coordinate (K) at (2.25,0) ; 
\coordinate (L) at (0.5,2) ; 
\coordinate (Q) at (2.474,0) ;
\end{scope}
\draw (O) circle(2cm) ; 
\draw[thick] (A) arc(0:120:2cm) ; 
\draw (O) node[left] {$O$} -- (A) node[right] {$A$} 
   -- (C) node[above left] {$C$} -- (R) node[below right] {$R$} ;
\draw (P) node[below left] {$P$} -- (D) node[right=2pt] {$D$}
   -- (B) node[above right] {$B$} ;
\draw (C) -- (B) -- (A) -- (M) node[below left] {$M$} ;
\draw (G) node[below right=-2pt] {$G$} 
   -- (H) node[below right=-2pt] {$H$}
   -- (I) node[below right=-2pt] {$I$} 
   -- (K) node[below right=-2pt] {$K$} -- (Q) node[right] {$Q$} ;
\draw (G) -- (L) node[above left] {$L$} -- (H) 
      (I) -- (L) -- (Q) ;
\draw[dashed] (L) -- (J) ;
\draw[gray] (M) ++(30:2mm) -- ++(120:2mm) -- ++ (-150:2mm) ;
\draw[gray] (D) ++(-120:2mm) -- ++(150:2mm) -- ++ (60:2mm) ;
\draw[gray] (J) ++(30:2mm) -- ++(120:2mm) -- ++ (-150:2mm) ;
\foreach \pt in {A,C,P,B,M,D,O,R}  \draw (\pt) node{$\bull$} ;
\foreach \mk in {G,H,I,K,L,Q}  \draw (\mk) node{$\sst\circ$} ;
\end{tikzpicture}
\caption{Estimating the arc length of a circular arc}
\label{fg:Huygens-XVI} 
\end{figure}

\begin{proof}
Take a circle with center $O$ and on it an arc $\arc{AC}$, less than a
semicircle. The sine of the arc $\arc{AC}$ is the perpendicular $AM$
from $A$ to the diameter $CR$. Let $B$ be the midpoint of the arc
$\arc{AC}$; the diameter $PB$ of the circle bisects the chord $AC$
(perpendicularly) at~$D$. (See Figure~\ref{fg:Huygens-XVI}.)

We already showed in the proof of Theorem~\ref{th:Huygens-VII}
(\emph{VII of Inventa}), see formula~\eqref{eq:Snell-any-arc}, that:
\begin{equation}
\arc{AC} > AC + \frac{1}{3}(AC - AM).
\label{eq:Huygens-XVI-first} 
\end{equation}
This establishes the first statement of the theorem.

\medskip

Now take another line $GH$, produced successively to points $I$, $K$
and~$Q$, such that:
\begin{itemize}
\item
$GH = AM$ and $GI = AC$, so that $AC - AM = GI - GH = HI$.
\item
$GK = GI + IK$, where $IK = \third HI$.
\item
$GQ = GI + IQ$, where by construction
$\dfrac{IQ}{IK} := \dfrac{4\,GI + GH}{2\,GI + 3\,GH}
= \dfrac{4\,AC + AM}{2\,AC + 3\,AM}$\,.
\end{itemize}

The sum relation $GQ = GI + IQ$ follows because $IQ > IK$, which in turn
comes from $4\,AC + AM > 2\,AC + 3\,AM$ since $AC > AM$. The 
placement of~$K$ also shows that 
$$
GK = AC + \frac{1}{3}(AC - AM).
$$
It follows from \eqref{eq:Huygens-XVI-first} that $\arc{AC} > GK$.

The second statement of the theorem can now be expressed as:
\begin{equation}
\arc{AC} < GQ 
= AC + \frac{1}{3}(AC - AM) \. \frac{4AC + AM}{2AC + 3AM}\,.
\label{eq:Huygens-XVI-second} 
\end{equation}

In order to prove \eqref{eq:Huygens-XVI-second}, construct the
triangles with common vertex $L$, common altitude equal to the radius
$OB$ of the circle, and respective bases $GH$, $HI$, and~$IQ$. Join
$OA$, $OC$, $AB$ and~$BC$ (see Figure~\ref{fg:Huygens-XVI} again).

Then, since
$$
\frac{IQ}{IK} = \frac{4\,GI + GH}{2\,GI + 3\,GH}
\implies \frac{IQ}{IH} = \frac{IQ}{3\,IK}
= \frac{4\,HI + GH}{6\,GI + 9\,GH}\,,
$$
it follows that
$$
\frac{QH}{IH} = \frac{QI + IH}{IH}
= \frac{10\,GI + 10\,GH}{6\,GI + 9\,GH}
= \frac{10}{3} \. \frac{GI + GH}{2\,GI + 3\,GH}
= \frac{10}{3} \. \frac{AC + AM}{2\,AC + 3\,AM}\,.
$$

Moreover, from the similar triangles
$\tri CAM \sim \tri COD$ we obtain
$$
\frac{AC}{AM} = \frac{OC}{OD}\,.
$$
Therefore,
$$
\frac{QH}{IH} = \frac{10}{3} \. \frac{OC + OD}{2\,OC + 3\,OD}
= \frac{10}{3} \. \frac{PD}{PB + 3\,OD}\,.
$$
Hence, by \eqref{eq:Huygens-XV} in the proof of
Theorem~\ref{th:Huygens-XV} (XV of \emph{Inventa}),
$$
\frac{(\segment ABC)}{(\tri ABC)} < \frac{QH}{IH}
= \frac{(\tri QHL)}{(\tri IHL)}\,.
$$

We now assert the equality of areas $(\tri IHL) = (\tri ABC)$.

Indeed, since the base $GH$ of the triangle $\tri GHL$ equals the
altitude $AM$ of $\tri OAC$ and vice~versa, we conclude that
$(\tri GHL) = (\tri OAC)$.

Moreover, since $GI = AC$ we similarly deduce that
$$
(\tri GIL) = (\tri OAB) + (\tri OBC) = (OABC),
$$
and then $(\tri IHL) = (\tri GIL) - (\tri GHL) 
= (OABC) - (\tri OAC) = (\tri ABC)$, as claimed. Therefore,
$$
\frac{(\segment ABC)}{(\tri ABC)} < \frac{(\tri QHL)}{(\tri ABC)}
\word{which implies}  (\tri QHL) > (\segment ABC).
$$
Moreover,
$$
(\tri QGL) = (\tri QHL) + (\tri GHL) 
> (\segment ABC) + (\tri OAC) = (\sector OAC).
$$

But the altitude of $\tri QGL$ equals the radius $OC$, by
construction. Therefore the \textit{base} $GQ$ of this triangle will
be greater than the \textit{arc} $\arc{AC}$. This establishes
\eqref{eq:Huygens-XVI-second} and proves Theorem~\ref{th:Huygens-XVI}.
\end{proof}

\begin{corl} 
\label{cr:Huygens-XVI}
Combining Theorems \ref{th:Huygens-VII} and~\ref{th:Huygens-XVI}, we
arrive at:
$$
\boxed{C_{2n} + \frac{C_{2n} - C_n}{3} < C < C_{2n}
+ \frac{C_{2n} - C_n}{3} \. \frac{4C_{2n} + C_n}{2C_{2n} + 3C_n}}
$$
where $C_n$ is the perimeter of an inscribed regular polygon of
$n$~sides and $C$ is the circumference of the circle.
\end{corl}


\subsection{Huygens' barycentric equation} 
\label{ssc:Huygens-barycenter-equation}
 
We return to Huygens' equation \eqref{eq:balancing-act}, restating it in
modern notation. We follow, to some extent, Hofmann's
treatment~\cite{Hofmann1966}.

In a circle of radius~$r$, consider a circular segment $ABC$, of
area~$\Sg < \half\pi r^2$. Choose a Cartesian coordinate system with
origin at the vertex $B$ of the segment and positive $x$-axis along
the diameter~$BD$ (see Figure~\ref{fg:Hofmann-five}). The equation of
the circle is
$$
y^2 = 2rx - x^2.
$$
The base of the segment lies on a line $x = a$ where $a < r$, and the 
endpoints of its arc are
$$
A = (a,\half b) \word{and} C = (a,-\half b), 
\word{where} \quarter b^2 = a(2r - a). 
$$
Let $L = (\xi,0)$ be the barycenter of the segment $ABC$ which, by
symmetry, lies on the diameter with $\xi < a$. The area of the
inscribed triangle may be denoted by
\begin{equation}
\dl := (\tri ABC) = \frac{ab}{2}\,.
\label{eq:inscribed-area} 
\end{equation}

\begin{figure}[htb]
\centering
\begin{tikzpicture}[scale=1.2, rotate=-90]
\coordinate (O) at (0,0) ;
\coordinate (A) at (110:2cm) ;
\coordinate (B) at (180:2cm) ;
\draw (B) ++(0,2) coordinate (Bf) ;
\coordinate (C) at (250:2cm) ;
\coordinate (D) at ($ (A)!0.5!(C) $) ;
\coordinate (G) at ($ (B)!0.5!(D) $) ;
\draw (G) +(0,2) coordinate (Gh) +(0,-2) coordinate (Gk) ;
\path[name path=ABarc] (A) arc(110:250:2cm) ; 
\path[name path=Gplus] (G) -- (Gh) ; 
\path[name path=Gminus] (G) -- (Gk) ; 
\path[name intersections={of=ABarc and Gplus, by={H}}] ;
\path[name intersections={of=ABarc and Gminus, by={K}}] ;
\coordinate (E) at ($ (A)!(H)!(D) $) ; 
\coordinate (F) at ($ (B)!(H)!(Bf) $) ; 
\coordinate (I) at ($ (B)!(K)!(Bf) $) ; 
\coordinate (J) at ($ (C)!(K)!(D) $) ; 
\coordinate (L) at ($ (B)!0.6!(D) $) ;  
\fill[blue!20] (A) arc(110:132:2cm) -- (E) -- cycle ;
\fill[blue!20] (C) arc(250:229:2cm) -- (J) -- cycle ;
\fill[gray!30] (B) arc(180:132:2cm) -- (F) -- cycle ;
\fill[gray!30] (B) arc(180:229:2cm) -- (I) -- cycle ;
\draw (A) arc(110:250:2cm) -- cycle ; 
\draw (E) node[below] {$E$} -- (H) node[above right] {$H$}
   -- (F) node[above right] {$F$} -- (B) node[above] {$B$}
   -- (I) node[above left] {$I$} --  (K) node[above left] {$K$}
   -- (J) node[below] {$J$} ;
\draw (H) -- (G) node[above left] {$G$} -- (K) ;
\draw (B) -- (L) node[below right=-2pt] {$L$} 
   -- (D) node[below] {$D$} ;
\draw[dashed] (A) node[below right] {$A$} -- (B) 
   -- (C) node[below left] {$C$} ;
\foreach \pt in {A,B,C,D,E,F,G,H,I,J,K}
   \draw (\pt) node {$\bull$} ;
\draw (L) node[red] {$\bull$};
\end{tikzpicture}
\caption{Estimating the segment's barycenter}
\label{fg:Hofmann-five} 
\end{figure}

\begin{lema}[Hofmann] 
\label{lm:Hofmann}
\begin{equation}
\xi > \frac{a}{2}\,.
\label{eq:barycenter-below-halfway} 
\end{equation}
\end{lema}

\begin{proof}
In the circular segment $ABC$ with base $AC$ and diameter $BD$, let
$G = (\half a,0)$ be the midpoint of $BD$ and draw the chord $HK$
through~$G$ parallel to~$AD$. Then $HK$ is the midline of a rectangle
$EFIJ$ whose base $EI$ is part of the base $AC$ of the segment, and
whose opposite side $FJ$ is tangent to the circle at~$B$.

The barycenter of this rectangle is~$G$. If we \textit{remove} from
the rectangle the regions $BFH$ and $BIK$ outside the circular
segment, and \textit{add} to the rectangle the regions $HAE$ and $KCJ$
inside the segment, we recover the segment $ABC$. Both of these moves
displace the barycenter from the midline $HK$ towards the base $AC$ of
the segment (i.e., downwards, in Figure~\ref{fg:Hofmann-five}). Thus
$L$ which, by symmetry, lies on the diameter $BD$, also lies below the
chord $HK$, i.e., nearer the base $AC$. In other words, $L = (\xi,0)$
lies on $BD$ between $G = (\half a,0)$ and~$D = (a,0)$. In brief,
$\half a < \xi < a$.
\end{proof}

We now restate Huygens' equation \eqref{eq:balancing-act}.

\begin{thm}[Huygens' barycentric equation] 
\label{th:barycentric-eqn}
\begin{equation}
\boxed{\frac{\Sg}{\dl} 
= \frac{2}{3} \. \frac{2r - a}{r - \xi}} \,.
\label{eq:barycentric-eqn} 
\end{equation}
\end{thm}

This is a very powerful result. For example, it leads quite rapidly to
a \textit{new proof} of the inequality~\eqref{eq:HH-first}, i.e., of
Lemma~\ref{lm:Huygens-III} (Theorem III of \emph{Inventa}). Indeed, if
we substitute \eqref{eq:barycenter-below-halfway} in the Huygens
barycentric equation we do recover~\eqref{eq:HH-first}:
$$
\frac{a}{2} < \xi \implies \frac{\Sg}{\dl} 
> \frac{2}{3} \. \frac{2r - 2\xi}{r - \xi} = \frac{4}{3} \,.
$$


\subsection{Hofmann's proof of XVI of \emph{Inventa}} 
\label{ssc:Hofmann}

In this subsection, another proof of Theorem~\ref{th:Huygens-XVI} is
given, following Hofmann \cite{Hofmann1966}, who in turn takes cues
from Schuh~\cite{Schuh1914}.

\begin{figure}[htb]
\centering
\begin{tikzpicture}[scale=1.2, rotate=-90]
\coordinate (A) at (100:2cm) ;
\draw (A) ++ (0,1) coordinate (A1) ;
\coordinate (B) at (180:2cm) ;
\coordinate (C) at (-100:2cm) ;
\draw (C) ++ (0,-1) coordinate (C1) ;
\coordinate (D) at ($ (A)!0.5!(C) $) ;
\coordinate (P) at ($ (B)!0.6!(D) $) ;
\draw (P) ++ (0,2) coordinate (P1) ;
\draw (P) ++ (0,-2) coordinate (P2) ;
\path[name path=bigOh] (B) circle(2cm) ;
\path[name path=Aup] (A) -- (A1) ;
\path[name path=Cdown] (C) -- (C1) ;
\path[name path=Pup] (P1) -- (P2) ;
\path[name intersections={of=Pup and bigOh, by={Q,R}}] ;
\draw[blue, name path=Para] 
    plot[variable=\y,smooth,domain=0:1.6]
    ({-2.0 + 0.66*\y^2},{1.4*\y}) ; 
\draw[blue, name path=Para] 
    plot[variable=\y,smooth,domain=0:1.59]
    ({-2.0 + 0.66*\y^2},{1.4*\y}) ; 
\draw[blue, name path=Paraleft] 
    plot[variable=\y,smooth,domain=0:1.59]
    ({-2.0 + 0.66*\y^2},{-1.4*\y}) ; 
\path[name intersections={of=Aup and Para, by={E,E1}}] ;
\path[name intersections={of=Cdown and Paraleft, by={F,F1}}] ;
\draw (A) arc(100:260:2cm) ; 
\draw (B) node[above] {$B$} -- (D) node[below] {$D$} ;
\draw (Q) node[above right] {$Q$} -- (P) node[above right] {$P$}
   -- (R) node[above left] {$R$} ;
\draw (E) node[above right] {$E$} -- (A) node[below] {$A$} 
   -- (C) node[below] {$C$} -- (F) node[above left] {$F$} ;
\foreach \pt in {A,B,C,D,E,F,P,Q,R} \draw (\pt) node{$\bull$} ;
\end{tikzpicture}
\caption{Comparing circular and parabolic segments}
\label{fg:Hofmann-six} 
\end{figure}

\begin{proof}[Hofmann's proof of Theorem~\ref{th:Huygens-XVI}]
Consider a circular segment $ABC$ with vertex $B$ and diameter $BD$.

Let $P$ be the point that divides the diameter in the ratio
$BP : PD = 3 : 2$. Draw the chord $QR$ through $P$, parallel to the 
base $AC$ of the segment, with $Q$ between $A$ and $B$, $R$ between 
$B$ and $C$, on the circular are $\arc{AC}$. Next, we construct the
\textit{parabola} with axis of symmetry $BD$ which passes through $B$,
$Q$ and~$R$, and cuts the prolongation of the line $AC$ at the points
$E$ and~$F$. (See Figure~\ref{fg:Hofmann-six}.)

As already argued in the proof of Theorem~\ref{th:Huygens-XIV} (see
Figure~\ref{fg:Huygens-XIV-two}), the \textit{parabolic} arcs
$\arc{BQ}$ and $\arc{BR}$ lie inside the segment $ABC$, while the
parabolic arcs $\arc{QE}$ and $\arc{RF}$ lie outside that segment.

By Proposition~II.8 of Archimedes' work \textit{On the Equilibrium of
Planes} \cite[pp.~214]{Heath1912}, the point $P$ is the barycenter of
the \textit{parabolic} segment $EBF$. To change this parabolic segment
into the \textit{circular} segment $ABC$, we make two modifications.
First, we \textit{add} the slivers $BQ$ and $BR$ bounded by the
circular and parabolic arcs, above the line $PQ$ (i.e., on the side
nearer~$B$). Second, we \textit{remove} the regions $AEQ$ and
respectively $CFR$ bounded by the parabolic arcs $\arc{EQ}$
(resp.~$\arc{FR}$), the circular arcs $\arc{QA}$ (resp.~$\arc{RC}$)
and the lines $AE$ (resp.~$CF$); these regions lie below the line~$PQ$
(i.e., on the side nearer~$D$). Both modifications shift the
barycenter upwards, towards~$B$. In this way, the barycenter $L$ of
the circular segment $ABC$, on the diameter~$BD$, is seen to lie
between $P$ and~$B$. Therefore,
$$
\xi < \frac{3a}{5} \,.
$$

Combining this relation with \eqref{eq:barycentric-eqn}, we obtain
$$
\frac{\Sg}{\dl} < \frac{4}{3} \. \frac{2r - a}{2r - 6a/5}\,.
$$

Now reconsider the circular segment $AEB$ in the notation of 
Figure~\ref{fg:Huygens-XVI}, with the parameters
\begin{equation}
BD = a, \quad 
AC = b, \quad 
AM = c, \quad 
OA = OB = OC = r, \quad  
OD = r - a,
\label{eq:segment-parameters} 
\end{equation}
and the arc length $s$ of~$\arc{AC}$.

The area \eqref{eq:inscribed-area} of the inscribed triangle 
$\tri ABC$ is $\dl = \half ab$. On the other hand, the similarity 
$\tri COD \sim \tri CAM$ implies that 
$$
\frac{r}{r - a} = \frac{b}{c}\,.
$$

Therefore, Huygens' barycentric equation \eqref{eq:barycentric-eqn}
gives us
$$
\Sg < \frac{4\dl}{3} \. \frac{2r - a}{2r - 6a/5}
= \frac{2ab}{3} \. \frac{b + c}{\frac{2}{5}(2b + 3c)}
= \frac{10}{3} \. \frac{b^2 - c^2}{2b + 3c} \. \frac{r}{2}\,.
$$
Adding the area $\half rc$ of the triangle $\tri OAC$, we obtain
$$
\frac{rs}{2} = (\sector OAC) = (\tri OAC) + \Sg 
< \frac{rc}{2} + \frac{10}{3}\.\frac{b^2 - c^2}{2b + 3c}\.\frac{r}{2}
$$
whereby, on dividing by $r/2$,
\begin{equation}
s < c + \frac{10}{3}\.\frac{b^2 - c^2}{2b + 3c}\,.
\label{eq:upper-bound} 
\end{equation}

This inequality \eqref{eq:upper-bound} is easily seen to be equivalent
to
$$
s < b + \frac{b - c}{3}\.\frac{4b + c}{2b + 3c}
$$
which is precisely Huygens' \eqref{eq:Huygens-XVI-second}.
\end{proof}

To compare \eqref{eq:upper-bound} with the one stated in the
introduction, the correspondence is:
\begin{equation}
s = x,  \quad
c = \sin x, \quad
b = 2\,AD = 2 \sin \frac{x}{2}\,.
\label{eq:trig-parameters} 
\end{equation}
Thereby, \eqref{eq:upper-bound} exactly recovers the modern formula
\eqref{eq:Huygens-upper}.

\paragraph{Error Analysis}
If we apply Corollary~\ref{cr:Huygens-XVI}, consistent with
\eqref{eq:upper-bound}, to a circle of diameter~$1$, we obtain
\begin{equation*}
\biggl\{ C_{2n} + \frac{C_{2n} - C_n}{3}
\. \frac{4C_{2n} + C_n}{2C_{2n} + 3C_n} \biggr\} - \pi
< \frac{\pi^7}{22400}\.\frac{1}{n^6} + O\Bigl( \frac{1}{n^8} \Bigr)
\end{equation*}
which shows the high accuracy of Huygens' approximation.

\begin{remk} 
\label{rk:area-comparison}
Although not needed for the proof, the diagram in 
Figure~\ref{fg:Hofmann-six} poses an intriguing question: which area 
is larger, the sliver $BQ$ or the triangular region $AEQ$? It turns 
out that they are almost equal, but the latter is indeed larger than 
the former in all cases. We examine this relation in 
Appendix~\ref{app:area-comparison}.
\end{remk}


\subsection{Huygens unproved barycentric theorem,
at the end of \textit{Inventa}} 
\label{ssc:Huygens-unfinished}

Huygens wanted a \textit{lower} bound with the same accuracy
$O(1/n^6)$ as his upper bound \eqref{eq:Huygens-XVI-second}. All he says,
in the final section (Problem~IV, alias Proposition~XX) of
\textit{Inventa}, is this \cite[p.~63]{Pearson1923}:
\begin{quote}
``[\dots] another lower limit may be obtained more accurate than the
first, if we use the following rule which depends upon a more careful
investigation of the center of gravity.''
\end{quote}
To achieve it he announced, \textit{without proof}, the following
approximation:

\begin{quote}
\itshape
Let four-thirds of the difference between the limits found be
added to twice the chord plus three times the sine, and let the
difference between the chord and the sine have the same ratio to
another line that the line thus made~up has to three and one-third or
ten-thirds times the sum of the two; this other line added to the sine
makes a line which is less than the~arc.
\end{quote}

To parse this in terms of our parameters
\eqref{eq:segment-parameters}, where the chord $AC$ has length~$b$ and
the sine $AM$ has length~$c$, the lower bound
\eqref{eq:Huygens-XVI-first} is expressed by
\eqref{eq:easy-lower-bound}:
$$
s > b + \frac{1}{3}(b - c) = c + \frac{4}{3}(b - c),
$$
so the two bounds differ by 
$$
\frac{10(b^2 - c^2)}{3(2b + 3c)} - \frac{4(b - c)}{3}
= \frac{2(b - c)^2}{3(2b + 3c)}\,.
$$
Thus, if
$$
d := 2b + 3c + \frac{8(b - c)^2}{9(2b + 3c)}
$$
then the new putative lower bound is $c + e$, where
$$
\frac{b - c}{e} = \frac{d}{\frac{10}{3}(b + c)}
$$
and hence, according to Huygens,
\begin{equation}
s > c + e = c + \frac{10}{3}\.
\frac{b^2 - c^2}{2b + 3c + \dfrac{8(b - c)^2}{9(2b + 3c)}}\,.
\label{eq:lower-bound} 
\end{equation}

We may combine these upper and lower bounds, showing their structural
similarity.

\begin{thm} 
\label{th:both-bounds}
The arc length $s$ of a circle segment with chord $b$ and sine~$c$ 
satisfies the bounds:
\begin{equation}
c + \frac{10}{3}\.
\frac{b^2 - c^2}{2b + 3c + \dfrac{8(b - c)^2}{9(2b + 3c)}}
< s < c + \frac{10}{3}\.\frac{b^2 - c^2}{2b + 3c}\,.
\label{eq:both-bounds} 
\end{equation}
\end{thm}

On converting the parameters of the inequalities
\eqref{eq:both-bounds} with the correspondence
\eqref{eq:trig-parameters}, we recover the formulas
\eqref{eq:Huygens-upper} and~\eqref{eq:Huygens-lower} stated in the
introduction. The theorem also gives a range of validity of those
formulas as $0 \leq x < \pi$. (The limiting case $x = \pi$ gives the
unremarkable estimates $\frac{30}{11} < \pi < \frac{10}{3}$.)
Naturally, the method of approximation by regular polygons shows that
the power of the inequalities lies in their application to
\emph{small} angles.

\begin{corl} 
\label{cr:both-bounds}
The circumference $C$ of a circle may be estimated in terms of the
perimeters $C_n$, $C_{2n}$ of inscribed regular polygons of $n$ and
$2n$ sides, respectively, as follows:
\begin{equation}
\boxed{C_n + \frac{10}{3} \. \frac{C_{2n}^2 - C_n^2}
{2C_{2n} + 3C_n 
+ \dfrac{8}{9}\.\dfrac{(C_{2n} - C_n)^2}{2C_{2n} + 3C_n}}
< C 
< C_n + \frac{10}{3} \. \frac{C_{2n}^2 - C_n^2}{2C_{2n} + 3C_n}}\,.
\label{eq:Huygens-bounds} 
\end{equation}
\end{corl}

\begin{proof}
If the chord $AC$ of Figure~\ref{fg:Huygens-XVI} is the side of a
regular polygon of $2n$~sides, its perimeter is $C_{2n} = 2nb$. The
sine $AM$ is \textit{half} the side length of a regular polygon of
$n$~sides (see Figure~\ref{fg:Huygens-VII}); hence its perimeter is
$C_n = 2nc$. On multiplying all terms of \eqref{eq:both-bounds}
by~$2n$, we obtain the global estimate \eqref{eq:Huygens-bounds}.
\end{proof}

For $n = 30$, Huygens used the following approximations
in~\eqref{eq:Huygens-bounds}:
\begin{align*}
3.13585389802979 &< C_{30} < 3.13585389802980,
\\
3.14015737457639 &< C_{60} < 3.14015737457640,
\\
\shortintertext{which gives}
3.14159265339060 &< \pi <3.14159265377520.
\end{align*}
This rounds to the now famous result, declared in~\emph{Inventa}:
\begin{equation}
\boxed{3.1415926533 < \pi < 3.1415926538}\,.
\label{eq:Huygens-for-pi} 
\end{equation}

Apparently, Huygens never wrote down a geometric proof of his new
lower bound~\eqref{eq:lower-bound}. To this day, such a
\emph{geometric} proof remains unknown.

\medskip

However, in 1914 Frederik Schuh \cite{Schuh1914} replaced a circular
segment by the cleverly chosen parabolic segment of
Figure~\ref{fg:Hofmann-six}, used Archimedes' determination of the
barycenter of the latter to prove the following inequality:
\begin{equation}
\xi > \frac{3}{5}\,a - \frac{3a^2}{25(r - \frac{3}{5}a)}\,.
\label{eq:Schuh} 
\end{equation}
Then Schuh used Huygens' barycentric equation
\eqref{eq:barycentric-eqn} to transform \eqref{eq:Schuh} into the
version of Huygens' final inequality with a constant~$27$ instead
of~$8$ in the denominator of Huygens' original inequality
\eqref{eq:lower-bound}. It is worth pointing out, as Schuh himself
noted \cite[p.~247]{Schuh1914}, that if we replace the $3$ in the
numerator of \eqref{eq:Schuh} by $\frac{8}{9}$, then his
transformation would produce Huygens' original lower bound. This is
the best result known, in a geometric formulation.

Of course, with differential calculus it is not difficult to verify
the inequality \eqref{eq:Huygens-lower}, equivalent to Huygens' lower
bound. Pinelis~\cite{Pinelis2024} suggested using the change of
variable $t := \tan(x/4)$ to transform the right-hand side into a
rational approximation $f(t)$ to $4 \arctan t$, which he then showed
to be an underestimate. Pinelis' argument seems to be the first
available proof of Huygens' lower bound.


\subsection{The location of the barycenter} 
\label{ssc:barycenter-location}

The distance from the center of a circle to the barycenter of a
circular segment is
\begin{equation}
\frac{4}{3}\. \frac{r \sin^3\frac{\th}{2}}{\th - \sin\th}
\label{eq:modern-location} 
\end{equation}
where $r$ is the radius of the circle and $\th$ is the central angle
of the segment. This formula can readily be obtained nowadays by
standard calculus techniques. Whether Huygens was cognizant of it is
not clear to us; but we now show that it is \emph{equivalent} to
Huygens' Theorem~V in \emph{Theoremata}, using only classical
geometrical arguments and the law of the lever.

That theorem, that we refer to as Huygens' Grossehilfsatz (quoted here
as our Proposition~\ref{pr:Huygens-Theoremata}) is the
\emph{Archimedean-style statement of the location of the barycenter of
a circular segment}. It is an \emph{implicit} statement, in that it
declares that the circular segment balances a certain triangle whose
barycenter is known (Archimedes located it at the intersection of the
medians), if the fulcrum is the center of the encompassing circle. It
does not explicitly give a formula for the barycenter; rather, it
locates it so that the equilibrium takes place.

We now show how to transform that implicit statement into the explicit
formula above.

\begin{thm} 
\label{th:modern-location}
According to Huygens' Proposition~\ref{pr:Huygens-Theoremata}, the 
distance from the center of the circle of radius~$r$ to the barycenter 
of the segment with central angle~$\th$ is given
by~\eqref{eq:modern-location}.
\end{thm}

\begin{proof}
Let $\bar x := OL$ be the distance from the center of the circle to
the barycenter of the segment $ABC$, see
Figure~\ref{fg:Huygens-Theoremata}, and let $a = BD$ be the diameter
and $b = AB$ the base of the segment, as
in~\eqref{eq:segment-parameters}. The Huygens associated triangle
$\tri KOH$ has base~$b$ and height $\sqrt{a(2r - a)} = \half b$, in
view of~\eqref{eq:counter-weight}. Therefore the distance from the
center $O$ of the circle to the barycenter $M$ of that triangle is
$\frac{2}{3}\.\half b = \third b$. Its area is
$(\tri KOH) = \quarter b^2$.

Now, since $\half b = r \sin \frac{\th}{2}$ and
$r - a = r \cos \frac{\th}{2}$, the area of the circular segment $ABC$
is
$$
\frac{1}{2} r^2\th - \frac{1}{2} b(r - a)
= \frac{1}{2} r^2\th - r^2 \sin\frac{\th}{2} \cos\frac{\th}{2}
= \frac{r^2}{2}(\th - \sin\th).
$$
Therefore, by the law of the lever,
$$
OM \. (\tri KOH) = OL \. (\segment ABC),
$$
that is,
\begin{equation}
\frac{b}{3} \. \frac{b^2}{4} = \bar x \. \frac{r^2}{2}(\th - \sin\th),
\label{eq:leverage} 
\end{equation}
which is to say,
$$
\frac{b^3}{12} = \frac{2}{3}\, r^3 \sin^3 \Bigr( \frac{\th}{2} \Bigr)
= \bar x \. \frac{r^2}{2}(\th - \sin\th).
$$
Solving for $\bar x$, we arrive at \eqref{eq:modern-location}:
$$
\bar x = \frac{4}{3}\,\frac{r \sin^3 \frac{\th}{2}}{\th - \sin\th}
$$
which is the known modern formula. 
\end{proof}

Clearly, the steps are reversible: since the last formula is 
equivalent to~\eqref{eq:leverage}, it expresses that the triangle 
$\tri KOH$ and the segment $ABC$ balance at the coordinate origin,
i.e., at the center $O$ of the circle.

Huygens solved the general problem of locating the barycenter of a
segment of \emph{any} conic section with a center (namely, an ellipse
or a hyperbola) in his earlier work \emph{Theoremata}. In this paper,
we adapt his results only to the special case of a circle. We surmise
that he was trying to complete the earlier work of Archimedes on
\emph{parabolic} segments.


\section{A historical conjecture} 
\label{sec:speculation}
 
Both the Cusa inequality and Snell's own inequality may be termed
\textit{convergence-improving} inequalities, since they produce at
least twice the accuracy as the original procedure of Archimedes.
 
We make the following historical conjecture: \textit{Archimedes was
fully cognizant of the Snell--Cusa convergence-improving inequalities
and may have used them to obtain closer bounds on~$\pi$}.

We base our suggestion on the following considerations.
\begin{enumerate}

\item 
It is well known that the extant text of \textit{Measurement} is a
damaged and corrupt extract from a more comprehensive study by
Archimedes of the metric properties of circular figures.
Unfortunately, his full tract was lost due to the willy-nilly of
history.

\item 
Some three centuries later, Heron \cite{Heron} quotes closer bounds on
$\pi$ computed by Archimedes, but they are corrupted (the quoted lower
bound is actually a quite good \textit{upper} bound). Moreover, he
gives no indication as to how they were calculated.
 
\item 
Heron also cites a proposition on circular segments from
\textit{Measurement} (namely, the \textit{first Heron--Huygens lemma})
which is \textit{not} in the extant text of today; indeed Eutokios
(6th century A.D.) had a version of \textit{Measurement}
\cite{Eutokios} which virtually coincides with today's variant. So by
then the original full treatise had already been lost.
 
\item 
Now we submit the major piece of evidence for our conjecture. The
treatise \textit{The Book of Lemmas}, transmitted by Arab
mathematicians, contains as its Proposition~8 Archimedes' famous
\textit{trisection of an angle} \cite[pp.~309]{Heath1912}. The
trisection construction creates a diagram (see
Figure~\ref{fg:Huygens-XII} above) \textit{which coincides exactly
with the Snell--Cusa figures in Huygens' treatise.} We submit that
this is no mere coincidence. Rather, it seems to us that Archimedes
worked assiduously on finding approximations to~$\pi$ and \textit{as
an incidental by-product} encountered the trisection construction. His
familiarity with the trisection figure and his investigations of such
approximations suggests to us that he must have been \textit{fully
aware of the applications of that figure to convergence-improving
approximations}. To suggest otherwise is to say that the most
brilliant mathematician of antiquity, working on geometric
approximations to~$\pi$, \textit{was unaware of that figure's
applications to his researches}. We consider this latter conclusion to
be inconceivable.
 
\item 
Assuming his familiarity with the Snell--Cusa improvements, it appears
quite likely that he used them to compute the closer bounds
cited by Heron. Our current knowledge precludes rigorous conclusions,
but the argument outlined above seems very plausible.

\end{enumerate}


\appendix

\section{Huygens' barycentric \emph{Grossehilfsatz}} 
\label{app:helpful}

The following argument establishes the auxiliary
Proposition~\ref{pr:Huygens-Theoremata}, needed for
Theorem~\ref{th:Huygens-Theoremata}.

\begin{figure}[htb]
\centering
\begin{tikzpicture}[scale=2.2]
\coordinate (O)  at (0,0) ;              
\coordinate (A)  at (1.732,1) ;          
\coordinate (B)  at (0,2) ;              
\coordinate (A1) at (0.433,1) ;          
\coordinate (A2) at (0.866,1) ;          
\coordinate (A3) at (1.299,1) ;          
\coordinate (C)  at (-1.732,1) ;         
\coordinate (C1) at (-1.299,1) ;         
\coordinate (C2) at (-0.866,1) ;         
\coordinate (C3) at (-0.433,1) ;         
\coordinate (D)  at (0,1) ;              
\coordinate (E) at (2,0) ;               
\coordinate (F) at (-2,0) ;              
\coordinate (G)  at (0,-1.732) ;
\coordinate (G1) at (-1.299,-1.732) ;    
\coordinate (G2) at (-0.866,-1.732) ;    
\coordinate (G3) at (-0.433,-1.732) ;    
\coordinate (G4) at (0.433,-1.732) ;     
\coordinate (G5) at (0.866,-1.732) ;     
\coordinate (G6) at (1.299,-1.732) ;     
\coordinate (H)  at (-1.732,-1.732) ;
\coordinate (H1) at (-1.299,-1.299) ;    
\coordinate (H2) at (-0.866,-0.866) ;    
\coordinate (H3) at (-0.433,-0.433) ;    
\coordinate (K) at (1.732,-1.732) ;
\coordinate (K1) at (1.299,-1.299) ;     
\coordinate (K2) at (0.866,-0.866) ;     
\coordinate (K3) at (0.433,-0.433) ;     
\coordinate (L) at ($ (O)!0.3!(D) $) ;   
\coordinate (N)  at (0,-1.85) ;          
\coordinate (P)  at (0,-2) ;             
\coordinate (P1) at (0,1.803) ;          
\coordinate (R1) at (-1.299,1.521) ;     
\coordinate (R2)  at (-0.866,1.803) ;    
\coordinate (R3) at (-0.433,1.953) ;     
\coordinate (R4) at (0.433,1.953) ;      
\coordinate (R5) at (0.866,1.803) ;      
\coordinate (R6) at (1.299,1.521) ;      
\coordinate (S)  at (0,-1.62) ;          
\coordinate (Y)  at (-1.082,0) ;         
\coordinate (Y1) at (-1.082,1.803) ;     
\coordinate (Y2) at (-1.082,1.401) ;     
\coordinate (Y3) at (-1.082,1) ;         
\coordinate (Y4) at (-1.082,-0.866) ;    
\coordinate (Y5) at (-1.082,-1.299) ;    
\coordinate (Y6) at (-1.082,-1.732) ;    
\draw (O) circle(2cm) ; 
\draw (A) node[right] {$A$} -- (B) node[above] {$B$}
    -- (C) node[left] {$C$} -- cycle ;
\draw (K) node[right] {$K$} -- (O) node[above right] {$O$}
    -- (H) node[above left] {$H$} -- cycle ;
\draw (B) -- (D) node[below left=-2pt] {$D$} -- (L) node[left] {$L$} 
    -- (S) node[above right=-2pt] {$S$}
    -- (G) node[left] {$G$} 
    -- (N) node[right=-2pt] {$N$} -- (P) node[below=2pt] {$E$} ;
\draw (E) node[right=2pt] {$J$} -- (F) node[left=2pt] {$I$} ;
\draw (Y1) node[above] {$F$} -- (Y2) node[left] {$V$}
    -- (Y3) node[below right] {$M$} -- (Y) node[above left] {$Y$}
    -- (Y4) node[above left] {$U$} -- (Y5) node[right] {$X$}
    -- (Y6) node[below=2pt] {$Z$} ;
\draw (P1) node[right] {$P$} -- (R2) node[below right] {$R$} ;
\draw[pattern={north west lines}, pattern color=blue!40]
      (C1) node[below left] {$Q$} rectangle (R2) ;
\draw[pattern={north west lines}, pattern color=blue!40]
      (G1) node[below left] {$T$} rectangle (H2) node[right] {$W$} ;
\draw[dashed] (C) rectangle (R1)  (C2) rectangle (R3) 
             (C3) rectangle (B) ;
\draw[dashed] (A) rectangle (R6)  (A3) rectangle (R5) 
             (A2) rectangle (R4)  (A1) rectangle (B) ;
\draw[dashed] (H) rectangle (H1)
             (G2) node[below right] {$T'$} rectangle (H3)
             (G3) rectangle (O) ;
\draw[dashed] (K) rectangle (K1)  (G6) rectangle (K2) 
             (G5) rectangle (K3)  (G4) rectangle (O) ;
\foreach \pt in {A,B,C,D,E,F,G,H,K,O,P,Y} 
    \draw (\pt) node{$\bull$} ;
\foreach \px in {C1,G1,G2,H2,P1,R2,Y1,Y2,Y3,Y4,Y5,Y6} 
    \draw (\px) node{$\bull$} ;
\foreach \mk in {L,N,S} \draw[red] (\mk) node{$\bull$} ;
\end{tikzpicture}
\caption{Balancing a circular segment and a triangle}
\label{fg:Huygens-barycenter} 
\end{figure}

Figure~\ref{fg:Huygens-barycenter}, adapted from~\cite{Huygens1651},
reproduces most of Figure~\ref{fg:Huygens-Theoremata} -- except for
the barycentric points $L$ and~$M$ of that diagram -- plus some
further detail. The circle segment $ABC$ and the triangle $KOH$ are as
before; they do not overlap. The emphasis now is on their disjoint
union:
$$
\sH := \segment ABC \cup \tri KOH.
$$
The segment $ABC$, less than a semicircle, is given; the diameter 
$BE$ of the circle extends the diameter $BD$ of the segment $ABC$.
The triangle $\tri KOH$ is constructed by first locating $G$ on the 
radius $OE$ so that
\begin{equation}
OG^2 = BD \. DE.
\label{eq:position-of-G} 
\end{equation}
The base $KH$ of the triangle $\tri KOH$ is drawn through~$G$, equal
and parallel to~$AC$, so that $G$ is its midpoint. The half-chord $CD$
is (by similar triangles, $\tri BCD \sim \tri CED$) the mean
proportional of $BD$ and~$DE$,
$$
BD \. DE = CD^2,
$$
consequently, $CD = OG$. 

Let $L$ now denote the barycenter of the combined figure
$\segment ABC \cup \tri KOH$. Proposition~\ref{pr:Huygens-Theoremata}
can now be restated as: \ $L = O$.

\begin{proof}
Suppose, \textit{arguendo}, that $L \neq O$. Then either $L$ is
``above'' $O$, that is, $L$ is an interior point of the radius~$OB$;
or $L$ is ``below'' $O$, i.e., $L$ is an interior point of the
radius~$OE$. We shall show that the first case is impossible; and by a
similar argument, the second case will also be ruled out.

Assume, then, that $L$ lies on $OB$ with $L \neq O$. Choose a
magnitude of area $\sM$ such that
$$
\frac{OG}{OL} = \frac{(\tri ABC \cup \tri KOH)}{\sM} \,,
$$
where we view the triangles as laminas of uniform density. Then cover
the combined figure $\sH$ with parallelograms of equal width,
symmetric with respect to the diameter~$BE$, such that if $\sR$ is the
total area of the \emph{excess parts} of the cover, then $\sR < \sM$.
That is possible since, by taking thinner parallelograms, \emph{we can
make the total excess area as small as we please}. Hence
$$
\frac{(\tri ABC) + (\tri KOH)}{\sR}
> \frac{(\tri ABC) + (\tri KOH)}{\sM} = \frac{OG}{OL}
$$
and therefore also:
$$
\frac{(\segment ABC) + (\tri KOH)}{\sR} > \frac{OG}{OL}\,.
$$

Now define the point $N$ on the prolongation of $BG$ \emph{below} the
base $KH$ such that
$$
\frac{ON}{OL} = \frac{(\segment ABC) + (\tri KOH)}{\sR}\,.
$$
Next, draw the diameter $IJ$ through $O$, parallel to the bases $AC$
and~$KH$. Take two \emph{corresponding parallelograms} $RQ$ and
$WT$ in the cover with \emph{respective barycenters}
$$
V \mot{in} RQ  \word{and}  X \mot{in} WT,
$$
and draw their midline $ZUMF$, cutting the line $IJ$ at~$Y$. Finally,
on the diameter $BE$ take the point $S$ such that $SE = BP$.

\begin{lema} 
\label{lm:barbells}
$Y$ is the barycenter of the combined parallelograms \ $RQ \cup WT$.
\end{lema}

\begin{proof}
By similar triangles,
$$
\frac{CD}{PR} = \frac{GH}{GT'} = \frac{OH}{OW} = \frac{ZY}{UY}\,.
$$
The numerator and denominator on the first of these ratios are 
mean proportionals:
$$
BD \. DE = CD^2  \word{and} BP \. PE = PR^2,
$$
and therefore
\begin{equation*}
\frac{BD \. DE}{BP \. PE} = \frac{CD^2}{PR^2} = \frac{ZY^2}{UY^2}\,.
\end{equation*}
It follows immediately that
\begin{equation}
\frac{BD \. DE}{BD \. DE - BP \. PE} 
= \frac{ZY^2}{ZY^2 - UY^2}\,.
\label{eq:quadratic-ratio} 
\end{equation}
The denominator on the left-hand side can be simplified thus:
\begin{align}
BD \. DE - BP \. PE
&= (BP + PD) \. DE -  BP \. (PD + DE) = PD \. (DE - BP)
\nonumber \\
&= (DE - SE) \. DP = DS \. DP,
\label{eq:diameter-trick} 
\end{align}
by the location of the point~$S$. 

Next, notice that
\begin{equation}
ZY^2 - UY^2 = ZU^2 + 2\,ZU \. UY
\label{eq:lower-rectangle} 
\end{equation}
as an instance of the algebraic identity
$(r + s)^2 - r^2 = s^2 + 2rs$. Since $X$ is the midpoint of the line
$ZU$, hence $ZX = XU = \half ZU$, the right-hand side
of~\eqref{eq:lower-rectangle} becomes
$$
4\,XU^2 + 4\,XU \. UY = 4\,XU \. (XU + UY) = 4\,XU \. XY,
$$
or equivalently, $ZU^2 + 2\,ZU \. UY = 2ZU \. XY$. Then
\eqref{eq:lower-rectangle} simplifies to
\begin{equation}
ZU^2 - UY^2 = 2ZU\. XY.
\label{eq:lower-rectangle-bis} 
\end{equation}

On substituting the relations \eqref{eq:diameter-trick} and
\eqref{eq:lower-rectangle-bis}, the
equality \eqref{eq:quadratic-ratio} reduces to
\begin{equation}
\frac{BD \. DE}{DS \. DP} = \frac{ZY^2}{2ZU \. XY}.
\label{eq:magic-ratio} 
\end{equation}

\medskip

Now, by the definition of the point $G$ using
\eqref{eq:position-of-G},
\begin{equation}
BD \. DE = OG^2 = ZY^2.
\label{eq:position-of-G-bis} 
\end{equation}
Moreover, since $SE = BP$ implies $OS = OP = OD + DP$, it follows that
$$
DS = OD + OS = 2\,OD + DP = 2\,YM + FM = 2(YM + MV) = 2\,YV.
$$
Therefore,
$$
DS \. DP = 2\,YV \. DP = 2YV \. FM.
$$
This, together with \eqref{eq:magic-ratio} and
\eqref{eq:position-of-G-bis}, allows us to conclude that
$YV \. FM = XY \. ZU$.

And that is equivalent to
$$
\frac{YV}{XY} = \frac{ZU}{FM} = \frac{(\pgram WT)}{(\pgram RQ)}\,.
$$
By the law of the lever, this means that $Y$ \emph{is the barycenter
of the two parallelograms $RQ \cup WT$ taken together}. This completes
the proof of the lemma.
\end{proof}

The same proof shows that the barycenter of \emph{every} pair of
corresponding parallelograms in the covering has its barycenter on the
line~$IJ$. Since the barycenter of the covering also lies on the
diameter~$BE$, it coincides with the center~$O$ of the circle.

\medskip

But now $L$ is the barycenter of the combined figure
$\segment ABC \cup \tri KOH$. Thus, the barycenter of just the excess
parts, of area $\sR$, is necessarily on the prolongation of the line
$LO$ such that
$$
\frac{\text{the added part}}{OL}
= \frac{(\segment ABC) + (\tri KOH)}{\sR} = \frac{ON}{OL} \,.
$$
The endpoint $N$ of this added part is therefore \emph{the barycenter
of the added part.}

But, \emph{that cannot be so!} Indeed, a line drawn through $N$
parallel to the base $KH$ of $\tri KOH$ leaves all of the area
portions which form $\sR$ on the opposite side of that line to~$N$,
which contradicts the nature of a barycenter.

\medskip

This rules out the first case, namely, that $L$ is a point of the
radius~$OB$, with $L \neq O$.

Were $L$ to be a point of the opposite radius~$OE$, with $L \neq O$, a
similar proof shows that the barycenter of $\sR$ would lie \emph{above
the segment $ABC$}, which again is absurd.

Therefore $L = O$, as claimed.
\end{proof}


\section{An area inequality} 
\label{app:area-comparison}

We may reconsider the comparison of the circular and parabolic 
segments in the second proof of Theorem~\ref{th:Huygens-XVI}. For 
that, we reproduce Figure~\ref{fg:Hofmann-six}, at a larger scale, in 
Figure~\ref{fg:overlap}.

\begin{figure}[htb]
\centering
\begin{tikzpicture}[scale=1.8, rotate=-90] 
\coordinate (O) at (0,0) ;
\coordinate (B) at (180:2cm) ;
\coordinate (A) at (100:2cm) ;
\draw (A) ++ (0,1) coordinate (A1) ;
\coordinate (C) at (-100:2cm) ;
\draw (C) ++ (0,-1) coordinate (C1) ;
\coordinate (D) at ($ (A)!0.5!(C) $) ;
\coordinate (P) at ($ (B)!0.6!(D) $) ;
\draw (P) ++ (0,2) coordinate (P1) ;
\draw (P) ++ (0,-2) coordinate (P2) ;
\path[name path=bigOh] (B) circle(2cm) ;
\path[name path=Aup] (A) -- (A1) ;
\path[name path=Cdown] (C) -- (C1) ;
\path[name path=Pup] (P1) -- (P2) ;
\path[name intersections={of=Pup and bigOh, by={Q,R}}] ;
\draw[blue, name path=Para] 
    plot[variable=\y,smooth,domain=0:1.59]
    ({-2.0 + 0.66*\y^2},{1.4*\y}) ; 
\draw[blue, name path=Paraleft] 
    plot[variable=\y,smooth,domain=0:1.59]
    ({-2.0 + 0.66*\y^2},{-1.4*\y}) ; 
\path[name intersections={of=Aup and Para, by={E,E1}}] ;
\path[name intersections={of=Cdown and Paraleft, by={F,F1}}] ;
\draw[->] (O) node[below left] {$\sst(0,b-r)$} node[below right] {$O$}
   -- (D) node[above left] {$\sst(0,0)$} node[above right] {$D$}
   -- (B) node[above right] {$\sst(0,b)$} node[above left] {$B$}
   -- ++(-0.5,0) ;
\draw (P) node[above left] {$\sst(0,\frac{2}{5} b)$} 
          node[above right] {$P$}
   -- (Q) node[above right=-2pt] {$\sst(p,\frac{2}{5} b)$} 
          node[below left] {$Q$} ;
\draw[->] (F) node[below left] {$\sst(-c,0)$} node[above left] {$F$}
   -- (C) node[below right] {$\sst(-a,0)$} node[below=2pt] {$C$}
   -- (A) node[below left] {$\sst(a,0)$} node[below=2pt] {$A$}
   -- (E) node[below right] {$\sst(c,0)$} node[above right] {$E$}
   -- ++(0,0.5) ;
\draw (A) arc(100:260:2cm) ; 
\draw (P) -- (R) node[below right] {$R$}
   node[above left=-2pt] {$\sst(-p,\frac{2}{5} b)$} ;
\draw[dashed] (A) -- (O) -- (C) ; 
\foreach \pt in {A,B,C,D,E,F,O,P,Q,R} \draw (\pt) node{$\bull$} ;
\end{tikzpicture}
\caption{Circular versus parabolic segments, again}
\label{fg:overlap} 
\end{figure}

We find it convenient to introduce Cartesian coordinates, as follows.
Place the origin $(0,0)$ at the midpoint $D$ of the chord $AC$ and
extend that chord to the $x$-axis; thus, $A = (a,0)$ and $C = (-a,0)$
for some $a > 0$. The vertex of the circle segment will lie at
$B = (0,b)$ for some $b > 0$.

If the circle has radius $r$, its center lies at $(0,b - r)$. Hence,
$$
a^2 + (r - b)^2 = r^2.
$$
The barycenter $P$ of the parabolic segment lies at $(0,\frac{2}{5}b)$
by Archimedes' theorem. The parabola and circle intersect, by 
construction, at the points $Q = (p,\frac{2}{5}b)$ and
$R = (-p,\frac{2}{5}b)$. The parabola cuts the $x$-axis (the extended
chord $AC$) at $E = (c,0)$ and $F = (-c,0)$.

The equation of the circle is $x^2 + (y + r - b)^2 = r^2$, or better,
in the semicircle including the arc $ABC$,
$$
y = \sqrt{r^2 - x^2} - r + b.
$$
The parabola with vertex at $B$, passing through $Q$ and~$R$, has the
equation
\begin{equation}
y = b - \frac{bx^2}{a^2 + \frac{2}{5}b^2} 
\equiv b - \frac{bx^2}{2rb - \frac{3}{5}b^2} \,.
\label{eq:parabola} 
\end{equation}
Since $(\pm p,\frac{2}{5}b)$ are the intersections of the parabola 
and the circle, we obtain
$$
p = \sqrt{\frac{3}{5}\,\Bigl( a^2 + \frac{2}{5}b^2 \Bigr)}
=: \frac{\sqrt{3}}{5}\,\sqrt{5a^2 + 2b^2}.
$$
The parabola \eqref{eq:parabola} cuts the $x$-axis at $(\pm c,0)$,
where
$$
c = \frac{1}{\sqrt{5}}\,\sqrt{10rb - 3b^2}.
$$

Here is a question raised by Figure~\ref{fg:overlap}: which area is
greater, the sliver $BQ$ between parabola and circle, or the curved
triangle $QAE$? If we also compare their mirror images, the sliver
$BR$ and the curved triangle $RCF$, we see that twice their difference
equals the difference in areas between the circular segment $ABC$ and
the parabolic segment $EBF$. Let us calculate those areas.

In his \textit{Quadrature of the Parabola}, Proposition~17, Archimedes
already computed the area of such a parabolic segment
\cite[pp.~246]{Heath1912} as $4/3$ the area if the inscribed triangle
with the same base and height. Therefore,
\begin{equation}
(\segment EBF) = \frac{4}{3}\,(\tri EBF) = \frac{4}{3}\,bc
= \frac{4b}{3\sqrt{5}}\,\sqrt{10rb - 3b^2}.
\label{eq:parabolic-area} 
\end{equation}

For the circular segment $ABC$, we subtract the triangle $\tri OAC$
from the \textit{sector} $OAC$. This triangle has base $AC = 2a$, and
altitude $OD = r - b$, so
$$
(\tri OAC) = a(r - b) = (r - b) \sqrt{2rb - b^2}.
$$
The sector has angle $\angle AOC =: \th$ with
$0 < \th < \pi$. From $\tri DOA$ 
one sees that $\cos(\half\th) = (r - b)/r$, so that 
$$
\frac{\th}{2} = \frac{\pi}{2} - \arcsin \frac{r - b}{r}\,.
$$
Now the area of the sector $OAC$ is
$$
(\sector OAC) = \frac{1}{2}\,r^2 \th 
= \frac{\pi r^2}{2} - r^2 \arcsin \frac{r - b}{r}
$$
and thus:
\begin{equation}
(\segment ABC) = \frac{\pi r^2}{2} - r^2 \arcsin \frac{r - b}{r} 
- (r - b) \sqrt{2rb - b^2}.
\label{eq:circular-area} 
\end{equation}
The difference in areas between the sliver~$BQ$ 
and the wedge $QAE$ is half the difference between 
\eqref{eq:circular-area} and
 \eqref{eq:parabolic-area}, namely:
$$
\frac{\pi r^2}{4} - \frac{1}{2} r^2 \arcsin \frac{r - b}{r} 
- \frac{r - b}{2}\,\sqrt{2rb - b^2} 
- \frac{2b}{3\sqrt{5}}\,\sqrt{10rb - 3b^2}.
$$

We claim that this quantity is \textit{negative} for all
$0 < b \leq r$. Since all its terms are areas, we may rescale by
setting $x := b/r$, $0 < x \leq 1$.

\begin{lema} 
The real function
$$
f(x) := \frac{\pi}{4} - \frac{1}{2} \arcsin(1 - x) 
- \frac{1}{2}\,(1 - x)\sqrt{2x - x^2} 
- \frac{2x}{3\sqrt{5}}\,\sqrt{10x - 3x^2}
$$
satisfies $f(x) < 0$ for $0 < x \leq 1$.
\end{lema}

\begin{proof}
Notice first that $f(0) = 0$. So it is enough to show that $f$ is
strictly decreasing in the interval $0 \leq x \leq 1$; which will
follow from $f'(x) < 0$ for $0 < x < 1$.

The derivative is easily computed:
$$
f'(x) = \sqrt{2x - x^2} 
- \frac{2x(5 - 2x)}{\sqrt{5}\,\sqrt{10x - 3x^2}}\,.
$$
This is a difference of two positive terms; so we must now check the 
following equivalent inequalities:
\begin{alignat*}{2}
\sqrt{2x - x^2} & < \frac{2x(5 - 2x)}{\sqrt{5}\,\sqrt{10x - 3x^2}}
&\word{for} & 0 < x < 1, 
\\
\sqrt{5(2x - x^2)(10x - 3x^2)} & < 2x(5 - 2x) &\word{for} & 0 < x < 1,
\\
\sqrt{5(2 - x)(10 - 3x)} & < 10 - 4x &\word{for} & 0 < x < 1, 
\\
5(2 - x)(10 - 3x) & < (10 - 4x)^2 &\word{for} & 0 < x < 1.
\end{alignat*}
But the last inequality comes immediately from
$$
(10 - 4x)^2 - 5(2 - x)(10 - 3x)
= (100 - 80x + 16x^2) - 5(20 - 16x + 3x^2) = x^2 > 0.
\eqno \qed
$$
\hideqed
\end{proof}

Therefore, $0 > f(x) > f(1)$ whenever $0 < x < 1$. Actually, the 
lower bound is quite small:
$$
f(1) = \frac{\pi}{4} - \frac{2\sqrt{35}}{15} 
\doteq 0.785398 - 0.788811 = -0.003413.
$$
Thus, in a circular segment less than a semicircle, the difference in
areas between wedge and sliver is less than about $r^2/300$.


\subsection*{Acknowledgments}

We thank Iosif Pinelis for discussions on MathOverflow
\cite{Pinelis2024} in regard to Huygens' final inequality. The authors
acknowledge the good offices of the Centro de Investigación en
Matemática Pura y Aplicada via a research project at the University of
Costa Rica.

\end{document}